\documentclass[a4paper,reqno]{amsart}
\usepackage[utf8]{inputenc}
\usepackage{amssymb}
\usepackage{amsmath}
\usepackage{mathtools}
\usepackage{ifpdf}
\usepackage{vmargin}
\usepackage{microtype}

\usepackage{setspace}

\ifpdf 
\usepackage[pdftex]{hyperref}
\else
\usepackage[hypertex]{hyperref}
\fi

\hypersetup{
	pdftitle={Low-regularity global well-posedness for the defocusing cubic wave equation outside a ball with radial initial data},
	pdfauthor={Ze Lin},
}

\theoremstyle{plain}
\newtheorem{theorem}{Theorem}[section]
\newtheorem{proposition}[theorem]{Proposition}
\newtheorem{lemma}[theorem]{Lemma}
\newtheorem{corollary}[theorem]{Corollary}

\theoremstyle{definition}
\newtheorem{definition}[theorem]{Definition}

\theoremstyle{remark}
\newtheorem{remark}[theorem]{Remark}

\numberwithin{equation}{section}

\newcommand{\R}{\mathbb{R}}

\newcommand{\Z}{\mathbb Z}

\newcommand{\bee}{\begin{eqnarray*}}
	\newcommand{\eee}{\end{eqnarray*}}

\allowdisplaybreaks[1]
\begin{document}
	\setstretch{1.12}
	
	\title[Low-regularity global well-posedness outside a ball]
	{Low-regularity global well-posedness\\
		for the defocusing cubic wave equation\\
		outside a ball with radial initial data}
	
	\date\today
	
	\author[Ze Lin]{Ze Lin}
	\address{\hskip-1.15em Ze Lin
		\hfill\newline Laboratory of Mathematics and Complex Systems,
		\hfill\newline Ministry of Education,
		\hfill\newline School of Mathematical Sciences,
		\hfill\newline Beijing Normal University,
		\hfill\newline Beijing, 100875, People's Republic of China.}
	\email{202421130044@mail.bnu.edu.cn}
	
	
	\begin{abstract}
		We study the three-dimensional defocusing cubic nonlinear wave equation
		with homogeneous Dirichlet boundary condition in the exterior of the unit
		ball, $\Omega=\R^3\backslash\overline{B(0,1)}$, for radial initial data.
		For every $\frac12\leq s\leq\frac34$, we establish global well-posedness
		in $\dot H^s_{D,rad}(\Omega)\times\dot H^{s-1}_{D,rad}(\Omega)$, including
		the critical endpoint $s=\frac12$, without an additional $L^4$ assumption
		on the initial displacement.
		The proof uses the distorted Fourier transform associated with the
		Dirichlet Laplacian to establish radial Sobolev embeddings, including
		the quantitative estimate
		$\|f\|_{L^p(\Omega)}\leq C\sqrt p\,\|f\|_{\dot H_D^{1/2}(\Omega)}$
		for $4\leq p<\infty$, with $C$ independent of $p$.
		We combine a nonlinear Fourier truncation argument with the critical
		Strichartz estimate to construct a global nonlinear component with
		high-frequency-truncated initial data.
		For the remainder, we construct a local energy solution in an affine
		energy space and derive a logarithmic energy inequality that prevents
		finite-time energy blow-up.
		Persistence of the original Sobolev regularity, uniqueness, and local
		Lipschitz dependence on each finite time interval complete the proof.
	\end{abstract}
	
	\maketitle

	\section{Introduction}
	
	We consider the three-dimensional defocusing cubic wave equation with
	homogeneous Dirichlet boundary condition outside the unit ball:
	\begin{equation}\label{eq:nlu}
		\begin{cases}
			\partial_t^2u-\Delta u+u^3=0,&(t,x)\in I\times\Omega,\\
			u(0,x)=u_0(x),\quad \partial_tu(0,x)=u_1(x),&x\in\Omega,\\
			
			u(t,x)=0,&x\in\partial\Omega,
		\end{cases}
	\end{equation}
	where $\Omega=\R^3\backslash\overline{B(0,1)}$, $I$ is an interval
	containing zero, and the initial data are real-valued and radial.
	We denote by $-\Delta_\Omega$ the Dirichlet Laplacian and use the associated
	homogeneous Sobolev spaces $\dot H_D^s(\Omega)$, recalled in
	Section~\ref{sect:DF transf}. Their radial subspaces are denoted by
	$\dot H^s_{D,rad}(\Omega)$.
	For sufficiently regular finite-energy solutions, the conserved energy is
	\begin{align}\label{def:energy}
		E(u)(t)
		&=\int_\Omega\left(\frac12|\partial_tu(t,x)|^2
		+\frac12|\nabla u(t,x)|^2+\frac14|u(t,x)|^4\right)\,dx
		\nonumber\\
		&=E(u)(0).
	\end{align}
	The corresponding whole-space equation is conformally invariant, and its
	scaling singles out the Sobolev index $s=\frac12$.
	Although scaling does not preserve the fixed exterior domain, we use the
	term critical regularity for this index, as in the whole-space problem.
	
	The low-regularity theory for nonlinear wave equations has been developed
	through Strichartz estimates, energy methods, and frequency decompositions;
	see, for example,
	\cite{BahChe:NLW,GallPlan03:NLW,KenigPV00:NLW,LindSog95:NLW,
		Roy:NLW:radial,Roy:NLW,Sogge95:book}.
	For the cubic equation on $\R^3$, Dodson proved global well-posedness and
	scattering for radial data in
	$\dot H^{1/2}(\R^3)\times\dot H^{-1/2}(\R^3)$
	\cite{Dod:NLW:sct1}.
	Related critical results for radial conformal wave equations in higher
	dimensions were obtained in \cite{Dod:conformal sharp}.
	These results concern the whole-space Cauchy problem; the presence of a
	boundary requires a linear theory and function spaces adapted to the
	Dirichlet Laplacian.
	
	For wave and Schr\"odinger equations on domains, dispersive estimates and
	Strichartz estimates depend on the geometry and the boundary condition.
	We refer to
	\cite{Burq03: NLW CPDE,BurqLP:NLW,BurqP:def NLW: Neum prob,
		DuyLaf:NLW:Neum prob,DuyYang24:LW Str,SmithSogge95:NLW JAMS,
		SmithSogge00:NLW CPDE}
	for representative results.
	In the radial exterior-ball setting, the explicit distorted Fourier
	transform provides a particularly useful description of the Dirichlet
	functional calculus; see \cite{LiSZ:NLS,LiXZ:NLS}.
	The multiplier theory, Littlewood--Paley estimates, and equivalence of
	Dirichlet and zero-extension homogeneous Sobolev norms are developed in
	\cite{KVZ16:IMRN Riesz tf}; see also
	\cite{IvPlan17:CPDE: heat flow,LiSZ:NLS}.
	
	For the specific problem \eqref{eq:nlu}, Xu and Yang
	\cite{XuYang:NLW} established radial dispersive and global-in-time
	Strichartz estimates and proved global well-posedness for
	\begin{align*}
		(u_0,u_1)\in
		\bigl(\dot H^s_{D,rad}(\Omega)\cap L^4(\Omega)\bigr)
		\times\dot H^{s-1}_{D,rad}(\Omega),
		\qquad \frac34<s<1.
	\end{align*}
	Their argument combines Fourier truncation with an energy estimate for a
	remainder equation.
	Xu, Yang, and You \cite{XuYangYou:NLW} subsequently treated the powers
	$3<p<5$ in the same exterior geometry.
	The purpose of the present paper is to establish global well-posedness
	for the cubic problem in the range $\frac12\leq s\leq\frac34$,
	including the critical endpoint, without an additional $L^4$ assumption
	on the initial displacement.
	
	We use the following notion of well-posedness.
	For a given radial initial pair, a mild solution satisfies the Duhamel
	formula associated with the Dirichlet wave propagators and belongs to
	\begin{align*}
		(u,\partial_tu)\in
		C\bigl(I;\dot H^s_{D,rad}(\Omega)
		\times\dot H^{s-1}_{D,rad}(\Omega)\bigr).
	\end{align*}
	Local well-posedness means existence on an interval containing zero,
	uniqueness in this continuous mild class, and continuous dependence on
	the initial data on a common local interval for nearby data.
	Global well-posedness means that the solution extends to all real times,
	with continuous dependence on every compact time interval.
	Our main result is the following.
	
	\begin{theorem}[Global well-posedness for radial data]
		\label{thm:radial global critical revised}\label{thm:nlu gwp}
		Let $\frac12\leq s\leq\frac34$.
		For every real-valued radial initial pair
		\begin{align*}
			(u_0,u_1)\in\dot H^s_{D,rad}(\Omega)
			\times\dot H^{s-1}_{D,rad}(\Omega),
		\end{align*}
		problem \eqref{eq:nlu} has a unique global mild solution satisfying
		\begin{align}\label{reg:main global revised}
			(u,\partial_tu)\in C\left(\R;
			\dot H^s_{D,rad}(\Omega)\times\dot H^{s-1}_{D,rad}(\Omega)\right).
		\end{align}
		More precisely,
		\begin{align}\label{eq:main global mild revised}
			u(t)={}&\cos(t\sqrt{-\Delta_\Omega})u_0
			+\frac{\sin(t\sqrt{-\Delta_\Omega})}{\sqrt{-\Delta_\Omega}}u_1
			\nonumber\\
			&-\int_0^t
			\frac{\sin((t-\tau)\sqrt{-\Delta_\Omega})}{\sqrt{-\Delta_\Omega}}
			u(\tau)^3\,d\tau,
		\end{align}
		with equality in $\dot H_D^s(\Omega)$.
		Uniqueness holds in the continuous mild class
		\eqref{reg:main global revised}.
		For every $T>0$, the data-to-solution map is locally Lipschitz from the
		initial-data space into
		\begin{align*}
			C\left([-T,T];\dot H^s_{D,rad}(\Omega)
			\times\dot H^{s-1}_{D,rad}(\Omega)\right).
		\end{align*}
		In addition, $u\in C(\R;L^4(\Omega))$.
	\end{theorem}
	
	\begin{remark}
		The additional $L^4$ condition is redundant in this range because
		\begin{align*}
			\dot H^s_{D,rad}(\Omega)\hookrightarrow L^4(\Omega),
			\qquad \frac14\leq s\leq\frac34.
		\end{align*}
		For $s>\frac12$, the Dirichlet boundary condition holds in the trace
		sense. At $s=\frac12$, it is understood through the critical Dirichlet
		Sobolev space and the spectral propagators in
		\eqref{eq:main global mild revised}; no ordinary pointwise boundary trace
		is asserted for general critical data.
		For $s>\frac12$, this trace assertion follows from the
		zero-extension equivalence in Theorem~\ref{thm:Sob eq} and the
		local trace theorem: multiplying the zero extension by a smooth
		cutoff supported near the boundary gives an $H^s(\R^3)$ function.
		Its trace is zero because it vanishes on the interior side of the
		sphere. At $s=\frac12$, Proposition~\ref{prop:Hardy for D-Lap}
		also gives the concrete weighted condition
		\begin{align*}
			\int_{1<|x|<2}\frac{|u(t,x)|^2}{|x|-1}\,dx
			\lesssim\|u(t)\|_{\dot H_D^{1/2}}^2<\infty.
		\end{align*}
		This condition does not assert the existence of an ordinary
		boundary trace at the critical index.
		The theorem concerns global well-posedness and does not assert scattering.
	\end{remark}
	
	We next describe the estimates and the decomposition used in the proof.
	For radial $f\in\dot H_D^s(\Omega)$, the embeddings needed below take the form
	\begin{align}\label{intro:radial embedding}
		\|f\|_{L^p(\Omega)}
		\lesssim_{s,p}\|f\|_{\dot H^s_D(\Omega)},
		\qquad 2\leq p<\infty,\quad
		\frac12-\frac1p\leq s\leq\frac32-\frac3p.
	\end{align}
	These embeddings can also be deduced from the radial weighted
	Hardy--Littlewood--Sobolev inequality of
	De N\'apoli, Drelichman, and Dur\'an \cite{DeNapoliDD:radial HLS},
	combined with the Sobolev norm equivalence in
	\cite{KVZ16:IMRN Riesz tf}.
	We include a direct proof using the distorted Fourier transform.
	For radial $f\in\dot H_D^{1/2}(\Omega)$, we also keep track of the dependence on $p$:
	\begin{align}\label{intro:critical embedding}
		\|f\|_{L^p(\Omega)}
		\leq C\sqrt p\,\|f\|_{\dot H_D^{1/2}(\Omega)},
		\qquad 4\leq p<\infty,
	\end{align}
	where $C$ is independent of $p$.
	After separating low and high spectral frequencies, the high-frequency
	estimate is consistent with the classical one-dimensional critical
	Sobolev and fractional Moser--Trudinger theory; see
	\cite{IulaMM:fractional MT,Ozawa:critical Sob}.
	The low-frequency part is controlled directly through the sine kernel
	of the distorted Fourier transform.
	
	We follow the Fourier truncation strategy of
	\cite{Bourgain98:FTM,GallPlan03:NLW,KenigPV00:NLW,XuYang:NLW}.
	For a sufficiently large, fixed $J$, write $u=w+v$, where $w$ solves
	the full nonlinear equation with initial pair
	$(P^\Omega_{>2^J}u_0,P^\Omega_{>2^J}u_1)$.
	This pair is small in
	$\dot H_D^{1/2}\times\dot H_D^{-1/2}$, including when $s=\frac12$,
	by convergence of the critical spectral tails.
	The global Strichartz estimates yield a global small solution $w$.
	The remainder has low-frequency initial data and satisfies
	\begin{align*}
		\partial_t^2v-\Delta_\Omega v+v^3=-3v^2w-3vw^2.
	\end{align*}
	In particular, $w$ is a nonlinear solution, and neither component is
	assumed to retain its initial spectral support.
	
	The remainder has finite initial energy, although its initial
	displacement need not belong to $L^2(\Omega)$.
	We therefore construct it locally by writing $v=v_0+h$ and solving for
	$(h,\partial_th)$ in $H_D^1(\Omega)\times L^2(\Omega)$.
	To extend this solution globally, we combine the radial energy-space
	bound $\|v\|_{L^\infty}\lesssim\|v\|_{\dot H_D^1}$ with
	\eqref{intro:critical embedding} and interpolate between $L^4$ and
	$L^\infty$.
	Writing $B=\sup_{t\in\R}\|w(t)\|_{\dot H_D^{1/2}}$, we obtain,
	for every $p\geq6$,
	\begin{align*}
		|E(v)'(t)|
		\lesssim B\sqrt p\,E(v)(t)^{1+1/p}+B^2E(v)(t).
	\end{align*}
	Choosing $p$ comparable to $\max\{6,\log(3+E(v)(t))\}$ yields a
	logarithmic differential inequality that excludes finite-time energy
	blow-up.
	This optimization is of Yudovich type; a related use of large-exponent
	estimates for cubic wave equations appears in the probabilistic setting
	of \cite[Section~5]{BurqTzvetkov:cubic}.
	Here the required estimates follow deterministically from radial symmetry
	and the exterior-ball geometry.
	Finally, Duhamel estimates recover the initial $s$-level regularity,
	and the local theory gives uniqueness and locally Lipschitz dependence
	on every finite time interval.
	
	The paper is organized as follows.
	Section~\ref{sect:DF transf} recalls the distorted Fourier transform and
	the associated function spaces, and establishes the radial embeddings
	used in the nonlinear analysis.
	Section~3 presents the linear estimates.
	Section~4 proves the local theory, constructs the nonlinear high-frequency
	component, and establishes the global energy bounds for the remainder,
	thereby proving Theorem~\ref{thm:nlu gwp}.
	
	\subsection*{Notation}\label{sect:notation}
	For nonnegative quantities $X$ and $Y$, the notation $X\lesssim Y$
	means $X\leq CY$ for some constant $C>0$, and $X\gtrsim Y$
	means $Y\lesssim X$. We write $X\approx Y$ when both inequalities
	hold. Subscripts indicate the parameters on which the constants may
	depend. Constants may change from line to line.
	We write $\Omega^c=\R^3\setminus\Omega$,
	$B(0,1)=\{x\in\R^3:|x|<1\}$, and
	$2^{\Z}=\{2^k:k\in\Z\}$. The notation
	$\operatorname{dist}(x,E)=\inf_{y\in E}|x-y|$ denotes distance
	to a nonempty set $E$. We use $C_c^\infty(\Omega)$ for smooth
	functions compactly supported in $\Omega$ and
	$C_{c,rad}^\infty(\Omega)$ for its radial subspace.
	The subscript $rad$ always indicates radial symmetry.
	The notation $X\hookrightarrow Y$ means a continuous embedding.
	The symbol $I$ in an operator expression denotes the identity
	operator; elsewhere $I$ denotes the time interval specified locally.
	We write $\mathbf1_E$ for the characteristic function of $E$, equal
	to one on $E$ and zero outside it, and
	$\operatorname{sgn}(\xi)$ for the sign of the real number $\xi$,
	with $\operatorname{sgn}(0)=0$.
	
	Unless otherwise specified, spatial norms are taken over $\Omega$
	with the standard measure $dx$; for radial functions,
	\begin{align*}
		\|f\|_{L^p(\Omega)}^p
		=4\pi\int_1^\infty |f(r)|^p r^2\,dr,
		\qquad 1\leq p<\infty.
	\end{align*}
	
	We abbreviate $\|f\|_p=\|f\|_{L^p(\Omega)}$ and use the mixed
	norm notation $L_t^qL_x^r(I\times\Omega)=L^q(I;L^r(\Omega))$.
	For finite $q$ and $r$,
	\begin{align*}
		\|f\|_{L_t^qL_x^r(I\times\Omega)}
		=\left(\int_I
		\left(\int_\Omega |f(t,x)|^r\,dx\right)^{q/r}dt
		\right)^{1/q},
	\end{align*}
	with the usual essential-supremum interpretation when an exponent
	is infinite. We write $L_{t,x}^q=L^q(I\times\Omega)$.
	The abbreviations $C_tX$, $L_t^qX$, and $C_bX$ refer to the time
	interval indicated in the surrounding argument.
	For $1<p<\infty$, $p'=p/(p-1)$ is the H\"older conjugate exponent.
	An intersection $X\cap Y$ is equipped with the sum norm.
	
	For a Banach space $X$, $C(I;X)$ denotes the space of continuous
	$X$-valued functions, and $C_b(I;X)$ its subspace of bounded continuous
	functions. On a compact interval these spaces carry the supremum norm.
	We use the unitary Euclidean Fourier convention
	\begin{align*}
		\widehat F(\xi)=(2\pi)^{-m/2}
		\int_{\R^m}e^{-ix\cdot\xi}F(x)\,dx,
		\qquad
		\|F\|_{\dot H^\sigma(\R^m)}^2
		=\int_{\R^m}|\xi|^{2\sigma}|\widehat F(\xi)|^2\,d\xi,
	\end{align*}
	initially on test functions, with the usual extensions to the
	Sobolev spaces used below. Here $m=1$ or $3$.
	We use the $L^2$ pairing
	$\langle f,g\rangle_{L^2}=\int_\Omega f\overline g\,dx$;
	the corresponding extended Sobolev duality is specified in
	Section~\ref{sect:DF transf}. Derivatives of rough solutions are
	understood in the distributional sense. All logarithms are natural.

	\section{Distorted Fourier transform and Dirichlet function spaces}
	\label{sect:DF transf}
	
	We collect the spectral representation and function-space results used
	below. The radial distorted Fourier transform is recalled from
	\cite{LiSZ:NLS,LiXZ:NLS}; the multiplier, Littlewood--Paley, Hardy, and
	Sobolev norm equivalence results are due to
	Killip, Visan, and Zhang~\cite{KVZ16:IMRN Riesz tf}.
	These results are used as background tools. We give the radial
	representation explicitly to fix the normalization in the subsequent
	embedding and nonlinear estimates.
	
	The Dirichlet Laplacian $-\Delta_\Omega$ is the nonnegative self-adjoint
	operator on $L^2(\Omega)$ associated with the quadratic form
	$\int_\Omega |\nabla f|^2\,dx$ with form domain $H_0^1(\Omega)$.
	Its operator domain is $H^2(\Omega)\cap H_0^1(\Omega)$.
	Here $H^k(\Omega)=W^{k,2}(\Omega)$, for $k=1,2$, is the space of
	$L^2$ functions whose distributional spatial derivatives up to order
	$k$ belong to $L^2$, and
	$H_0^1(\Omega)=\overline{C_c^\infty(\Omega)}^{\,H^1(\Omega)}$.
	
	For a radial test function $f$, set
	\begin{align*}
		h(\rho)=(1+\rho)f(1+\rho),\qquad \rho>0.
	\end{align*}
	Then
	\begin{align*}
		(1+\rho)(-\Delta_\Omega f)(1+\rho)=-h''(\rho),
		\qquad h(0)=0,
	\end{align*}
	and
	\begin{align*}
		\|f\|_{L^2(\Omega)}^2=4\pi\int_0^\infty |h(\rho)|^2\,d\rho.
	\end{align*}
	Thus the radial Dirichlet Laplacian is unitarily equivalent to the
	Dirichlet Laplacian on the half-line. In particular, its radial spectrum
	is purely absolutely continuous and equals $[0,\infty)$, with no
	$L^2$ eigenvalues.
	
	The corresponding generalized eigenfunctions are
	\begin{equation}\label{quan:eign funt}
		e_\lambda(r)=\frac{\sin(\lambda(r-1))}{r},
		\qquad \lambda>0,\quad r\geq1.
	\end{equation}
	They satisfy $-\Delta_\Omega e_\lambda=\lambda^2e_\lambda$ and
	$e_\lambda(1)=0$. These real-valued modes contain both incoming and
	outgoing components. The parameter $\lambda$ is the nonnegative square
	root of the spectral variable; integration over $\lambda>0$ already
	gives the full sine-transform representation. Zero is a spectral
	threshold and does not contribute an $L^2$ eigenfunction or a discrete
	term to this representation.
	
	For $f\in C_{c,rad}^\infty(\Omega)$, define
	\begin{align}\label{def:DF}
		\mathcal F_Df(\lambda)
		&=\sqrt{\frac2\pi}\int_1^\infty
		e_\lambda(r)f(r)r^2\,dr\nonumber\\
		&=\sqrt{\frac2\pi}\int_0^\infty
		\sin(\lambda\rho)h(\rho)\,d\rho,
		\qquad \lambda>0.
	\end{align}
	The last expression is the unitary sine transform of $h$.
	Consequently,
	\begin{align}\label{Id:resolution}
		\int_1^\infty |f(r)|^2r^2\,dr
		=\int_0^\infty |\mathcal F_Df(\lambda)|^2\,d\lambda.
	\end{align}
	The transform therefore extends from test functions to an isomorphism
	from $L^2_{rad}(\Omega)$ onto $L^2(0,\infty)$, and
	$(4\pi)^{1/2}\mathcal F_D$ is unitary for the standard measure $dx$.
	Its inverse is
	\begin{align}\label{def:inverse DF}
		\mathcal F_D^{-1}g(r)
		=\sqrt{\frac2\pi}\frac1r\int_0^\infty
		\sin(\lambda(r-1))g(\lambda)\,d\lambda.
	\end{align}
	For general $g\in L^2(0,\infty)$, this formula is interpreted as an
	$L^2$ limit of spectral truncations, rather than as an absolutely
	convergent integral. The sine-transform inversion theorem yields
	\begin{align}\label{Id:resol2}
		\mathcal F_D^{-1}\mathcal F_Df=f,
		\qquad \mathcal F_D\mathcal F_D^{-1}g=g
	\end{align}
	in the respective $L^2$ spaces.
	
	For a bounded Borel multiplier $m$, the radial functional calculus is
	given by
	\begin{align*}
		m(\sqrt{-\Delta_\Omega})f
		=\mathcal F_D^{-1}\bigl(m(\lambda)\mathcal F_Df(\lambda)\bigr).
	\end{align*}
	This identity also specifies the multipliers on the spectral Sobolev
	spaces used below. Multiplication of functions in physical space does
	not in general preserve, or add, their distorted spectral supports.
	
	\begin{theorem}[Multiplier theorem; \cite{KVZ16:IMRN Riesz tf}]
		\label{thm:multiplier}
		Let $m\in C^2((0,\infty))$ satisfy
		\begin{align*}
			|\partial_\lambda^k m(\lambda)|\leq C_k\lambda^{-k},
			\qquad k=0,1,2.
		\end{align*}
		Then $m(\sqrt{-\Delta_\Omega})$ extends to a bounded operator on
		$L^p(\Omega)$ for every $1<p<\infty$.
	\end{theorem}
	
	Fix a smooth nonincreasing function $\phi:[0,\infty)\to[0,1]$ with
	\begin{align}\label{def:phi}
		\phi(\lambda)=1\quad(0\leq\lambda\leq1),
		\qquad \phi(\lambda)=0\quad(\lambda\geq2).
	\end{align}
	For $N\in2^{\Z}$, set
	\begin{align}\label{def:phi scal}
		\phi_N(\lambda)=\phi(\lambda/N),\qquad
		\psi_N(\lambda)=\phi_N(\lambda)-\phi_{N/2}(\lambda),
	\end{align}
	and define
	\begin{align*}
		P^\Omega_{\leq N}f&=\phi_N(\sqrt{-\Delta_\Omega})f,\\
		P^\Omega_Nf&=\psi_N(\sqrt{-\Delta_\Omega})f,\\
		P^\Omega_{>N}f&=(I-P^\Omega_{\leq N})f,\\
		\widetilde P^\Omega_Nf
		&=P^\Omega_{N/2}f+P^\Omega_Nf+P^\Omega_{2N}f.
	\end{align*}
	Here $\sum_{N\in2^{\Z}}\psi_N(\lambda)=1$ for $\lambda>0$.
	
	\begin{definition}\label{def:Besov}
		For $s\in\R$, $1<q<\infty$, and $1\leq r<\infty$, we use the
		homogeneous Dirichlet Besov norm
		\begin{align*}
			\|f\|_{\dot B^s_{D,q,r}(\Omega)}
			=\left(\sum_{N\in2^{\Z}}
			N^{sr}\|P^\Omega_N f\|_{L^q(\Omega)}^r\right)^{1/r}.
		\end{align*}
		The space is obtained by completion of finite spectral sums of test
		functions under this norm. More precisely, the initial class consists
		of finite sums of functions $P^\Omega_N\varphi$ with
		$\varphi\in C_c^\infty(\Omega)$. Only the cases $s=0$ and $r=2$
		are used in the Strichartz argument below.
	\end{definition}
	
	\begin{lemma}[Bernstein estimates; \cite{KVZ16:IMRN Riesz tf,LiSZ:NLS}]
		\label{lem:berns ests}
		For radial $f\in C_c^\infty(\Omega)$ and $s\in\R$,
		\begin{align}
			\|P^\Omega_{\leq N}f\|_{L^p}
			+\|P^\Omega_Nf\|_{L^p}
			&\lesssim\|f\|_{L^p},\label{est:LP bd1}\\
			N^s\|P^\Omega_Nf\|_{L^p}
			&\approx\|(-\Delta_\Omega)^{s/2}P^\Omega_Nf\|_{L^p},
			\label{est:LP bd2}
		\end{align}
		for $1\leq p\leq\infty$. Moreover,
		\begin{align*}
			\|P^\Omega_{\leq N}f\|_{L^q}
			+\|P^\Omega_Nf\|_{L^q}
			\lesssim N^{3(1/p-1/q)}\|f\|_{L^p},
			\qquad 1\leq p\leq q\leq\infty.
		\end{align*}
		The constants are independent of $N$. For $1<p<\infty$,
		$f=\sum_{N\in2^{\Z}}P^\Omega_Nf$ with convergence in $L^p(\Omega)$.
	\end{lemma}
	
	\begin{definition}\label{def:Sob DB}
		For $s\geq0$ and $1<p<\infty$, let $\dot H_D^{s,p}(\Omega)$ and
		$H_D^{s,p}(\Omega)$ be the completions of $C_c^\infty(\Omega)$ under
		\begin{align*}
			\|f\|_{\dot H_D^{s,p}}&=\|(-\Delta_\Omega)^{s/2}f\|_{L^p},\\
			\|f\|_{H_D^{s,p}}&=\|(I-\Delta_\Omega)^{s/2}f\|_{L^p}.
		\end{align*}
		For $p=2$, we write $\dot H_D^s$ and $H_D^s$.
		The subscript $rad$ denotes the radial subspace.
	\end{definition}
	
	Integration by parts on radial test functions gives
	\begin{align*}
		\mathcal F_D(-\Delta_\Omega f)(\lambda)
		=\lambda^2\mathcal F_Df(\lambda).
	\end{align*}
	With the standard measure $dx$, the spectral Sobolev norm is therefore
	\begin{align}\label{eq:radial spectral Sob norm}
		\|f\|_{\dot H_D^s(\Omega)}^2
		=4\pi\int_0^\infty\lambda^{2s}
		|\mathcal F_Df(\lambda)|^2\,d\lambda.
	\end{align}
	In particular, $(4\pi)^{1/2}\mathcal F_D$ identifies the radial
	homogeneous energy space with $L^2((0,\infty),\lambda^2\,d\lambda)$.
	For the negative indices used in this paper, we define
	$\dot H^{-t}_{D,rad}(\Omega)$, $0<t\leq\frac34$, by Hilbert-space
	duality with $\dot H^t_{D,rad}(\Omega)$ relative to the extended $L^2$
	pairing. Equivalently, these spaces have the weighted spectral norm
	\eqref{eq:radial spectral Sob norm} with $s=-t$.
	For a general element of a spectral Sobolev space, $\mathcal F_D f$
	is understood by this extension; no absolutely convergent integral
	representation is asserted.
	More explicitly, for $0<t\leq\frac34$, the extended pairing satisfies
	\begin{align*}
		\langle F,\varphi\rangle
		&=4\pi\int_0^\infty
		\mathcal F_DF(\lambda)
		\overline{\mathcal F_D\varphi(\lambda)}\,d\lambda,\\
		\|F\|_{\dot H_D^{-t}}
		&=\sup_{\substack{\varphi\in C_{c,rad}^\infty(\Omega)\\
				\|\varphi\|_{\dot H_D^t}\leq1}}
		|\langle F,\varphi\rangle|.
	\end{align*}
	The integral is well defined by weighted Cauchy--Schwarz.
	These identities follow from the weighted $L^2$ spectral realization
	and the density of radial test functions in $\dot H_D^t$.

	\begin{proposition}[Square-function estimate; {\cite[Theorem~4.3]{KVZ16:IMRN Riesz tf}}]
		\label{prop:squ funt}
		For $1<p<\infty$, $s\geq0$, and $f\in C_c^\infty(\Omega)$,
		\begin{align*}
			\|(-\Delta_\Omega)^{s/2}f\|_{L^p}
			\approx
			\left\|\left(\sum_{N\in2^{\Z}}
			N^{2s}|P^\Omega_N f|^2\right)^{1/2}\right\|_{L^p}.
		\end{align*}
	\end{proposition}
	
	We invoke this estimate from the cited reference. In particular,
	$\dot B^s_{D,2,2}$ and $\dot H_D^s$ have equivalent norms in the
	range used here.
	
	\begin{proposition}[{\cite[Lemma~4.4]{KVZ16:IMRN Riesz tf}}]
		\label{prop:dense}
		Let $1<p<\infty$ and $0\leq s<1+\frac1p$. Then
		\begin{align*}
			\{(-\Delta_\Omega)^{s/2}f:f\in C_c^\infty(\Omega)\}
		\end{align*}
		is dense in $L^p(\Omega)$.
	\end{proposition}
	
	\begin{remark}
		This is a density statement for the range of $(-\Delta_\Omega)^{s/2}$
		on test functions. It does not assert that the homogeneous space
		$\dot H_D^{s,p}(\Omega)$ is a subspace of $L^p(\Omega)$.
	\end{remark}
	
	\begin{definition}\label{def:Sob 0 extd}
		Let $\dot H_{00}^{s,p}(\Omega)$ and $H_{00}^{s,p}(\Omega)$ denote the
		completions of the zero extensions of $C_c^\infty(\Omega)$ in
		$\dot H^{s,p}(\R^3)$ and $H^{s,p}(\R^3)$, respectively.
	\end{definition}
	
	\begin{proposition}[Hardy inequality; \cite{KVZ16:IMRN Riesz tf}]
		\label{prop:Hardy for D-Lap}
		Let $1<p<\infty$ and $0<s<\min\{1+\frac1p,\frac3p\}$.
		For $f\in C_c^\infty(\Omega)$,
		\begin{align*}
			\left\|\frac{f}{\operatorname{dist}(x,\Omega^c)^s}\right\|_{L^p}
			\lesssim_{s,p}\|(-\Delta_\Omega)^{s/2}f\|_{L^p}.
		\end{align*}
	\end{proposition}
	
	\begin{remark}
		The restrictions in this inequality reflect both the behavior near
		the boundary and the behavior at infinity; see
		\cite{KVZ16:IMRN Riesz tf} for sharpness and counterexamples.
	\end{remark}
	
	\begin{theorem}[Sobolev norm equivalence; \cite{KVZ16:IMRN Riesz tf}]
		\label{thm:Sob eq}
		Let $1<p<\infty$ and $0\leq s<\min\{1+\frac1p,\frac3p\}$.
		For $f\in C_c^\infty(\Omega)$ and its zero extension $\widetilde f$,
		\begin{align*}
			\|(-\Delta_\Omega)^{s/2}f\|_{L^p(\Omega)}
			\approx_{s,p}
			\|(-\Delta_{\R^3})^{s/2}\widetilde f\|_{L^p(\R^3)}.
		\end{align*}
		Thus $\dot H_D^{s,p}(\Omega)$ and $\dot H_{00}^{s,p}(\Omega)$
		are identified by zero extension, with equivalent norms.
	\end{theorem}
	
	\begin{remark}\label{rem:Frac prod}
		Theorem~\ref{thm:Sob eq} allows the transfer of Euclidean fractional
		product estimates to the Dirichlet setting, provided all Sobolev
		exponents lie in its range; see \cite{KVZ16:IMRN Riesz tf}.
		For the nonlinear estimates below we instead use the radial $L^4$
		embeddings and their duals, so no fractional product estimate is needed.
	\end{remark}
	
	\subsection{Radial Sobolev embeddings}
	
	The qualitative embeddings below can also be obtained from the radial
	weighted fractional integration inequality of
	\cite{DeNapoliDD:radial HLS} and Theorem~\ref{thm:Sob eq}.
	We include a proof in the distorted Fourier representation to separate
	the low and high spectral frequencies and to obtain the quantitative
	estimate used in the energy argument.
	
	\begin{proposition}[Radial Sobolev embeddings]\label{prop:radial Sob embedding}
		Let $\Omega=\R^3\backslash\overline{B(0,1)}$, $2\leq p<\infty$, and
		\begin{align}\label{cond:radial Sob embedding}
			\frac12-\frac1p\leq s\leq\frac32-\frac3p.
		\end{align}
		Then, for every $f\in\dot H^s_{D,rad}(\Omega)$,
		\begin{align}\label{est:radial Sob embedding}
			\|f\|_{L^p(\Omega)}
			\lesssim_{s,p}\|f\|_{\dot H^s_{D,rad}(\Omega)}.
		\end{align}
		In particular,
		\begin{align}\label{est:radial L4 embedding}
			\|f\|_{L^4(\Omega)}
			\lesssim_s\|f\|_{\dot H^s_{D,rad}(\Omega)},
			\qquad \frac14\leq s\leq\frac34.
		\end{align}
	\end{proposition}
	
	\begin{proof}
		We first suppose that $f\in C^\infty_{c,rad}(\Omega)$.
		Recall that
		\begin{align*}
			e_\lambda(r)&=\frac{\sin\lambda(r-1)}r,\\
			\mathcal F_D f(\lambda)
			&=\sqrt{\frac2\pi}\int_1^\infty e_\lambda(r)f(r)r^2\,dr,\\
			\mathcal F_D^{-1}g(r)
			&=\sqrt{\frac2\pi}\frac1r
			\int_0^\infty\sin\lambda(r-1)g(\lambda)\,d\lambda.
		\end{align*}
		The Plancherel identity and the spectral functional calculus give
		\eqref{eq:radial spectral Sob norm}.
		Here and below, the $L^p(\Omega)$ norms use the standard measure $dx$.
		Using the cutoff $\phi$ from \eqref{def:phi}, decompose
		\begin{align*}
			f=P^\Omega_{\leq1}f+P^\Omega_{>1}f.
		\end{align*}
		
		\noindent\textit{Low frequencies.}
		Set
		\begin{align*}
			g_{\leq1}(\lambda)&=\phi(\lambda)\mathcal F_Df(\lambda),\\
			h_{\leq1}(\rho)&=\sqrt{\frac2\pi}
			\int_0^\infty\sin(\lambda\rho)g_{\leq1}(\lambda)\,d\lambda,
			\qquad \rho>0.
		\end{align*}
		The inverse distorted Fourier transform gives
		\begin{align*}
			P^\Omega_{\leq1}f(1+\rho)=\frac{h_{\leq1}(\rho)}{1+\rho}.
		\end{align*}
		Define an auxiliary radial function on $\R^3$ by
		\begin{align*}
			F_{\leq1}(x)=\frac{h_{\leq1}(|x|)}{|x|}.
		\end{align*}
		This is a shifted radial reconstruction, not the zero extension of
		$P^\Omega_{\leq1}f$.
		Since $p\geq2$, we have $(1+\rho)^{2-p}\leq\rho^{2-p}$, and hence
		\begin{align}\label{est:low radial comparison}
			\|P^\Omega_{\leq1}f\|_{L^p(\Omega)}^p
			&=4\pi\int_0^\infty|h_{\leq1}(\rho)|^p
			(1+\rho)^{2-p}\,d\rho\nonumber\\
			&\leq4\pi\int_0^\infty|h_{\leq1}(\rho)|^p
			\rho^{2-p}\,d\rho
			=\|F_{\leq1}\|_{L^p(\R^3)}^p.
		\end{align}
		Let $\sigma=\frac32-\frac3p$.
		The usual three-dimensional Sobolev inequality yields
		\begin{align*}
			\|F_{\leq1}\|_{L^p(\R^3)}
			\lesssim_p\|F_{\leq1}\|_{\dot H^\sigma(\R^3)}.
		\end{align*}
		For $p=2$, this is simply the $L^2$ identity.
		The radial sine-transform representation of the whole-space Laplacian gives
		\begin{align*}
			\|F_{\leq1}\|_{\dot H^\sigma(\R^3)}^2
			=4\pi\int_0^\infty\lambda^{2\sigma}
			|g_{\leq1}(\lambda)|^2\,d\lambda.
		\end{align*}
		Since $g_{\leq1}$ is supported in $[0,2]$, the condition $s\leq\sigma$ implies
		\begin{align*}
			\|F_{\leq1}\|_{\dot H^\sigma(\R^3)}^2
			&\leq4\pi\,2^{2(\sigma-s)}
			\int_0^\infty\lambda^{2s}|\mathcal F_Df(\lambda)|^2\,d\lambda\\
			&\lesssim_{s,p}\|f\|_{\dot H_D^s(\Omega)}^2.
		\end{align*}
		Consequently,
		\begin{align}\label{est:low radial Sob}
			\|P^\Omega_{\leq1}f\|_{L^p(\Omega)}
			\lesssim_{s,p}\|f\|_{\dot H_D^s(\Omega)},
			\qquad s\leq\frac32-\frac3p.
		\end{align}
		
		\noindent\textit{High frequencies.}
		Set
		\begin{align*}
			g_{>1}(\lambda)&=(1-\phi(\lambda))\mathcal F_Df(\lambda),\\
			h_{>1}(\rho)&=\sqrt{\frac2\pi}
			\int_0^\infty\sin(\lambda\rho)g_{>1}(\lambda)\,d\lambda.
		\end{align*}
		Then
		\begin{align}\label{est:high radial comparison}
			\|P^\Omega_{>1}f\|_{L^p(\Omega)}^p
			&=4\pi\int_0^\infty|h_{>1}(\rho)|^p
			(1+\rho)^{2-p}\,d\rho\nonumber\\
			&\leq4\pi\|h_{>1}\|_{L^p(0,\infty)}^p.
		\end{align}
		Let $H_{>1}$ be the odd extension of $h_{>1}$ to $\R$:
		\begin{align*}
			H_{>1}(\rho)=
			\begin{cases}
				h_{>1}(\rho),&\rho>0,\\
				-h_{>1}(-\rho),&\rho<0,\\
				0,&\rho=0.
			\end{cases}
		\end{align*}
		With the Fourier convention fixed above, sine-transform inversion
		gives
		\begin{align*}
			\widehat H_{>1}(\xi)
			=-i\operatorname{sgn}(\xi)g_{>1}(|\xi|)
		\end{align*}
		for almost every $\xi\ne0$. Let $\alpha=\frac12-\frac1p$.
		The one-dimensional Sobolev inequality therefore gives
		\begin{align*}
			\|h_{>1}\|_{L^p(0,\infty)}
			&\lesssim_p\|H_{>1}\|_{\dot H^\alpha(\R)}\\
			&\lesssim_p\left(\int_1^\infty
			\lambda^{2\alpha}|g_{>1}(\lambda)|^2\,d\lambda\right)^{1/2}.
		\end{align*}
		For $p=2$, use the Plancherel identity instead.
		Since $g_{>1}$ vanishes on $[0,1]$, if $s\geq\alpha$, then
		\begin{align*}
			\int_1^\infty\lambda^{2\alpha}|g_{>1}(\lambda)|^2\,d\lambda
			\leq\int_0^\infty\lambda^{2s}|\mathcal F_Df(\lambda)|^2\,d\lambda.
		\end{align*}
		Thus
		\begin{align}\label{est:high radial Sob}
			\|P^\Omega_{>1}f\|_{L^p(\Omega)}
			\lesssim_{s,p}\|f\|_{\dot H_D^s(\Omega)},
			\qquad s\geq\frac12-\frac1p.
		\end{align}
		Combining \eqref{est:low radial Sob} and \eqref{est:high radial Sob}
		proves \eqref{est:radial Sob embedding} for smooth compactly supported radial functions.
		
		Finally, Definition~\ref{def:Sob DB} defines the Dirichlet Sobolev spaces
		by completion. Radial approximants may be obtained by rotation averaging,
		which commutes with the Dirichlet functional calculus and is contractive
		in the $L^2$-based spectral norm.
		For $f\in\dot H^s_{D,rad}(\Omega)$, choose
		$f_n\in C^\infty_{c,rad}(\Omega)$ with
		$\|f_n-f\|_{\dot H_D^s}\to0$.
		The established estimate implies
		\begin{align*}
			\|f_n-f_m\|_{L^p(\Omega)}
			\lesssim_{s,p}\|f_n-f_m\|_{\dot H_D^s(\Omega)}\longrightarrow0.
		\end{align*}
		Their $L^p$ limit defines the continuous embedding of the completion,
		and the same estimate follows by passage to the limit.
		Taking $p=4$ gives \eqref{est:radial L4 embedding}, including both endpoints.
	\end{proof}
	
	The high-frequency estimate in the following proposition is an
	instance of one-dimensional critical Sobolev growth; compare
	\cite{IulaMM:fractional MT,Ozawa:critical Sob}.
	We prove the required exterior-domain estimate directly, including
	the low-frequency contribution.
	
	\begin{proposition}[Quantitative critical embedding]\label{prop:critical sqrt p}
		There exists $C>0$, independent of $p$, such that, for every
		$f\in\dot H^{1/2}_{D,rad}(\Omega)$,
		\begin{align}\label{est:critical sqrt p}
			\|f\|_{L^p(\Omega)}
			\leq C\sqrt p\,\|f\|_{\dot H^{1/2}_{D,rad}(\Omega)},
			\qquad 4\leq p<\infty.
		\end{align}
	\end{proposition}
	
	\begin{proof}
		It suffices first to consider $f\in C^\infty_{c,rad}(\Omega)$.
		Use the same low- and high-frequency decomposition as above and put
		\begin{align*}
			B=\left(\int_0^\infty\lambda|\mathcal F_Df(\lambda)|^2\,d\lambda\right)^{1/2}
			=(4\pi)^{-1/2}\|f\|_{\dot H_D^{1/2}(\Omega)}.
		\end{align*}
		For $\rho\geq0$, Cauchy--Schwarz gives
		\begin{align*}
			|P^\Omega_{\leq1}f(1+\rho)|
			&\lesssim\frac B{1+\rho}
			\left(\int_0^2\frac{\sin^2(\lambda\rho)}\lambda\,d\lambda\right)^{1/2}\\
			&\lesssim B\frac{\sqrt{\log(2+\rho)}}{1+\rho}.
		\end{align*}
		For $\rho\geq1$, the last bound follows by splitting the integral at
		$\lambda=\rho^{-1}$ and using $|\sin z|\leq\min\{|z|,1\}$;
		for $0\leq\rho\leq1$, use $|\sin(\lambda\rho)|\leq\lambda\rho$.
		Since
		\begin{align*}
			\sup_{\rho\geq0}\frac{\sqrt{\log(2+\rho)}}{1+\rho}<\infty,
			\qquad
			\int_0^\infty\frac{\log^2(2+\rho)}{(1+\rho)^2}\,d\rho<\infty,
		\end{align*}
		we obtain
		\begin{align*}
			\|P^\Omega_{\leq1}f\|_{L^4(\Omega)}
			+\|P^\Omega_{\leq1}f\|_{L^\infty(\Omega)}\lesssim B.
		\end{align*}
		Interpolation yields
		\begin{align}\label{est:critical low uniform p}
			\|P^\Omega_{\leq1}f\|_{L^p(\Omega)}\lesssim B,
			\qquad 4\leq p<\infty,
		\end{align}
		with a constant independent of $p$.
		
		For the high-frequency part, \eqref{est:high radial comparison}
		and the one-dimensional Hausdorff--Young inequality give
		\begin{align*}
			\|P^\Omega_{>1}f\|_{L^p(\Omega)}
			\lesssim\|g_{>1}\|_{L^{p'}(1,\infty)},
			\qquad p'=\frac p{p-1},
		\end{align*}
		with a constant uniform for $p\geq4$.
		H\"older's inequality implies
		\begin{align*}
			\|g_{>1}\|_{L^{p'}(1,\infty)}
			&\leq B\left(\int_1^\infty
			\lambda^{-p/(p-2)}\,d\lambda\right)^{(p-2)/(2p)}\\
			&=B\left(\frac{p-2}{2}\right)^{(p-2)/(2p)}
			\lesssim\sqrt p\,B.
		\end{align*}
		Together with \eqref{est:critical low uniform p}, this proves
		\eqref{est:critical sqrt p} for smooth radial functions.
		The completion argument in Proposition~\ref{prop:radial Sob embedding}
		extends it to $\dot H^{1/2}_{D,rad}(\Omega)$ for each finite $p$.
	\end{proof}
	
	\begin{remark}
		For $\frac12\leq s\leq\frac34$, the additional displacement assumption
		$f\in L^4(\Omega)$ is redundant in the radial class:
		\begin{align*}
			\dot H^s_{D,rad}(\Omega)\cap L^4(\Omega)
			=\dot H^s_{D,rad}(\Omega),
			\qquad
			\|f\|_{\dot H_D^s}+\|f\|_{L^4}
			\approx_s\|f\|_{\dot H_D^s}.
		\end{align*}
		The quantitative estimate \eqref{est:critical sqrt p} does not assert
		an embedding of $\dot H^{1/2}_{D,rad}(\Omega)$ into $L^\infty(\Omega)$.
	\end{remark}
	
	\section{Linear estimates}\label{sect:linear est}
	
	We record the linear estimates used in the nonlinear argument.
	The critical Strichartz estimate below is a consequence of the
	global-in-time radial estimates of Xu and Yang
	\cite[Theorem~3.3]{XuYang:NLW}.
	We use their result as an input and only explain its formulation in
	the Lebesgue and spectral Sobolev spaces needed here.
	The remaining estimates follow directly from the Dirichlet functional
	calculus recalled in Section~\ref{sect:DF transf}.
	
	Let $I\subset\R$ be an interval containing zero and consider
	\begin{equation}\label{eq:line wave}
		\begin{cases}
			\partial_t^2u-\Delta_\Omega u=F,&(t,x)\in I\times\Omega,\\
			u(0,x)=u_0(x),\quad\partial_tu(0,x)=u_1(x),&x\in\Omega,\\
			u(t,x)=0,&x\in\partial\Omega.
		\end{cases}
	\end{equation}
	Its mild formulation is
	\begin{align}\label{lin wave}
		u(t)={}&\cos(t\sqrt{-\Delta_\Omega})u_0
		+\frac{\sin(t\sqrt{-\Delta_\Omega})}{\sqrt{-\Delta_\Omega}}u_1
		\nonumber\\
		&+\int_0^t
		\frac{\sin((t-\tau)\sqrt{-\Delta_\Omega})}{\sqrt{-\Delta_\Omega}}
		F(\tau)\,d\tau.
	\end{align}
	For rough data the propagators are defined spectrally.
	When the source is measured only in a Strichartz dual space, the
	Duhamel term is defined by continuous extension from smooth sources.
	
	\subsection{Spectral energy estimates}
	
	\begin{proposition}[Spectral energy estimate]
		\label{prop:linear spectral energy revised}
		Let $\frac14\leq s\leq1$ and let the data and source be radial, with
		\begin{align*}
			(u_0,u_1)&\in\dot H^s_{D,rad}(\Omega)
			\times\dot H^{s-1}_{D,rad}(\Omega),\\
			F&\in L^1\bigl(I;\dot H_D^{s-1}(\Omega)\bigr).
		\end{align*}
		Then \eqref{lin wave} defines a continuous pair in
		$\dot H_D^s\times\dot H_D^{s-1}$ and
		\begin{align}\label{est:linear spectral energy revised}
			&\|u\|_{L^\infty(I;\dot H_D^s)}
			+\|\partial_tu\|_{L^\infty(I;\dot H_D^{s-1})}
			\nonumber\\
			&\qquad\lesssim
			\|u_0\|_{\dot H_D^s}+\|u_1\|_{\dot H_D^{s-1}}
			+\|F\|_{L^1(I;\dot H_D^{s-1})}.
		\end{align}
		The implicit constant is independent of $I$.
	\end{proposition}
	
	\begin{proof}
		By the distorted Fourier representation,
		\begin{align*}
			\mathcal F_D\bigl(\cos(t\sqrt{-\Delta_\Omega})f\bigr)(\lambda)
			&=\cos(t\lambda)\mathcal F_Df(\lambda),\\
			\mathcal F_D\left(
			\frac{\sin(t\sqrt{-\Delta_\Omega})}{\sqrt{-\Delta_\Omega}}g
			\right)(\lambda)
			&=\frac{\sin(t\lambda)}{\lambda}\mathcal F_Dg(\lambda).
		\end{align*}
		The bound $|\sin(t\lambda)|,|\cos(t\lambda)|\leq1$ implies
		\begin{align*}
			\|\cos(t\sqrt{-\Delta_\Omega})f\|_{\dot H_D^s}
			&\leq\|f\|_{\dot H_D^s},\\
			\left\|\frac{\sin(t\sqrt{-\Delta_\Omega})}
			{\sqrt{-\Delta_\Omega}}g\right\|_{\dot H_D^s}
			&\leq\|g\|_{\dot H_D^{s-1}},\\
			\|\sqrt{-\Delta_\Omega}\sin(t\sqrt{-\Delta_\Omega})f
			\|_{\dot H_D^{s-1}}
			&\leq\|f\|_{\dot H_D^s},\\
			\|\cos(t\sqrt{-\Delta_\Omega})g\|_{\dot H_D^{s-1}}
			&\leq\|g\|_{\dot H_D^{s-1}}.
		\end{align*}
		For example, with the standard three-dimensional measure,
		\begin{align*}
			\left\|\frac{\sin(t\sqrt{-\Delta_\Omega})}
			{\sqrt{-\Delta_\Omega}}g\right\|_{\dot H_D^s}^2
			&=4\pi\int_0^\infty\lambda^{2s-2}
			|\sin(t\lambda)|^2|\mathcal F_Dg(\lambda)|^2\,d\lambda\\
			&\leq\|g\|_{\dot H_D^{s-1}}^2.
		\end{align*}
		Apply these estimates to \eqref{lin wave} and its time derivative,
		and use the triangle inequality for Banach-space integrals.
		Strong continuity follows by dominated convergence in the spectral
		variable and the $L^1$ integrability of the source.
	\end{proof}
	
	\subsection{The critical Strichartz estimate}
	
	\begin{theorem}[Critical Strichartz estimate;
		\cite{XuYang:NLW}]\label{thm:St est}
		Let $u_0,u_1$ and $F$ be radial in the spatial variable. If
		\begin{align*}
			(u_0,u_1)&\in\dot H^{1/2}_{D,rad}(\Omega)
			\times\dot H^{-1/2}_{D,rad}(\Omega),\\
			F&\in L^{4/3}(I\times\Omega),
		\end{align*}
		then the solution of \eqref{eq:line wave} satisfies
		\begin{align}\label{est:critical linear revised}
			&\|u\|_{L^4(I\times\Omega)}
			+\|u\|_{L^\infty(I;\dot H_D^{1/2})}
			+\|\partial_tu\|_{L^\infty(I;\dot H_D^{-1/2})}
			\nonumber\\
			&\qquad\leq C\left(
			\|u_0\|_{\dot H_D^{1/2}}
			+\|u_1\|_{\dot H_D^{-1/2}}
			+\|F\|_{L^{4/3}(I\times\Omega)}\right).
		\end{align}
		Moreover, $(u,\partial_tu)$ is continuous in time with values in
		$\dot H_D^{1/2}\times\dot H_D^{-1/2}$.
		The constant $C$ is independent of $I$; in particular, the estimate
		holds on $I=\R$.
	\end{theorem}
	
	\begin{proof}
		In \cite[Theorem~3.3]{XuYang:NLW}, choose
		\begin{align*}
			q_1=q_2=r_1=r_2=4,\qquad
			\rho_1=\rho_2=0,\qquad \mu=\frac12.
		\end{align*}
		These parameters satisfy the admissibility and regularity conditions
		of that theorem. Its estimate controls the solution in
		$L_t^4\dot B^0_{D,4,2}$ and the continuous critical energy class,
		with source norm $L_t^{4/3}\dot B^0_{D,4/3,2}$.
		The square-function estimates imply the continuous embeddings
		\begin{align*}
			\dot B^0_{D,4,2}(\Omega)&\hookrightarrow L^4(\Omega),\\
			L^{4/3}(\Omega)&\hookrightarrow\dot B^0_{D,4/3,2}(\Omega).
		\end{align*}
		Indeed, for smooth functions the square-function equivalence and
		Minkowski's inequality give
		\begin{align*}
			\|f\|_4
			&\lesssim
			\left(\sum_N\|P_N^\Omega f\|_4^2\right)^{1/2},\\
			\left(\sum_N\|P_N^\Omega f\|_{4/3}^2\right)^{1/2}
			&\lesssim
			\left\|\left(\sum_N|P_N^\Omega f|^2\right)^{1/2}\right\|_{4/3}
			\lesssim\|f\|_{4/3}.
		\end{align*}
		Combining these embeddings with the cited theorem yields
		\eqref{est:critical linear revised}.
	\end{proof}
	
	\begin{remark}
		The estimate used here is the critical $L^4_{t,x}$ Strichartz estimate,
		not the additional radial $L_t^2L_x^q$ estimate in
		\cite[Theorem~3.5]{XuYang:NLW}.
		The latter is not needed for the logarithmic energy argument below.
		The dispersive estimates and their proofs are available in
		\cite[Section~3]{XuYang:NLW} and are not reproduced here.
	\end{remark}
	
	\subsection{A local inhomogeneous energy estimate}
	
	For the local construction of the remainder, we also need the
	inhomogeneous energy space. Recall that
	\begin{align*}
		H_D^1(\Omega)=H_0^1(\Omega),\qquad
		\|f\|_{H_D^1}^2=\|f\|_2^2+\|\sqrt{-\Delta_\Omega}f\|_2^2.
	\end{align*}
	
	\begin{proposition}\label{prop:linear local H1 revised}
		Let $0<T\leq1$, $h_0\in H_D^1(\Omega)$,
		$h_1\in L^2(\Omega)$, and
		$F\in L^1([-T,T];L^2(\Omega))$.
		Then the linear solution $h$ with initial pair $(h_0,h_1)$ satisfies
		\begin{align}\label{est:linear local H1 revised}
			&\|h\|_{C([-T,T];H_D^1)}
			+\|\partial_th\|_{C([-T,T];L^2)}
			\nonumber\\
			&\qquad\leq C\left(
			\|h_0\|_{H_D^1}+\|h_1\|_2
			+\|F\|_{L^1([-T,T];L^2)}\right),
		\end{align}
		where $C$ is independent of $T\leq1$.
	\end{proposition}
	
	\begin{proof}
		Proposition~\ref{prop:linear spectral energy revised} with $s=1$
		controls the homogeneous energy norms in the radial case; the same
		spectral argument works without radial symmetry.
		To control the additional $L^2$ norm, use
		\begin{align*}
			\|\cos(t\sqrt{-\Delta_\Omega})h_0\|_2&\leq\|h_0\|_2,\\
			\left\|\frac{\sin(t\sqrt{-\Delta_\Omega})}
			{\sqrt{-\Delta_\Omega}}g\right\|_2&\leq |t|\|g\|_2.
		\end{align*}
		The second estimate follows from
		$|\sin(t\lambda)|\leq|t|\lambda$.
		Applying it to the initial velocity and the Duhamel term gives
		\begin{align*}
			\sup_{|t|\leq T}\|h(t)\|_2
			\leq\|h_0\|_2+T\|h_1\|_2
			+T\|F\|_{L^1([-T,T];L^2)}.
		\end{align*}
		Since $T\leq1$, the desired estimate follows.
	\end{proof}
	
	\section{Global well-posedness: Proof of Theorem \ref{thm:nlu gwp}}\label{sect:gwp}
	
	In this section, we prove Theorem~\ref{thm:nlu gwp}.
	We combine the critical Strichartz estimate of
	Xu and Yang~\cite{XuYang:NLW} with the radial embeddings in
	Section~\ref{sect:DF transf}.
	Following the nonlinear Fourier truncation framework of
	\cite{Bourgain98:FTM,GallPlan03:NLW,KenigPV00:NLW,XuYang:NLW},
	we construct a small global nonlinear solution from the high-frequency
	initial data. The remaining component is controlled by a logarithmic
	energy estimate.
	Throughout the global argument, we assume that
	$\frac12\leq s\leq\frac34$ and
	\begin{align*}
		(u_0,u_1)\in\dot H^s_{D,rad}(\Omega)
		\times\dot H^{s-1}_{D,rad}(\Omega).
	\end{align*}
	The local result below holds on the larger range
	$\frac14\leq s\leq\frac34$.
	
	\subsection{Nonlinear decomposition and small critical tails}\label{subs:w H1/2 GWP}
	
	Following the nonlinear Fourier truncation framework in
	\cite{GallPlan03:NLW,XuYang:NLW}, we first construct the nonlinear
	component associated with the high-frequency part of the initial data.
	Throughout this subsection, assume that
	$\frac12\leq s\leq\frac34$ and
	\begin{align*}
		(u_0,u_1)\in
		\dot H^s_{D,rad}(\Omega)\times\dot H^{s-1}_{D,rad}(\Omega).
	\end{align*}
	For an integer $J\geq0$, set
	\begin{align*}
		w_0=P^\Omega_{>2^J}u_0,\qquad
		w_1=P^\Omega_{>2^J}u_1,
	\end{align*}
	and consider
	\begin{equation}\label{eq:nlw}
		\begin{cases}
			\partial_t^2w-\Delta_\Omega w+w^3=0,
			&(t,x)\in\R\times\Omega,\\
			w(0,x)=w_0(x),\quad\partial_tw(0,x)=w_1(x),
			&x\in\Omega,\\
			w(t,x)=0,
			&x\in\partial\Omega.
		\end{cases}
	\end{equation}
	Since $P^\Omega_{>2^J}$ is defined through the Dirichlet spectral
	calculus, the truncated data belong to the corresponding radial
	Dirichlet Sobolev spaces. The boundary condition is understood in
	the sense specified for the main theorem.
	Only the initial data are spectrally localized; no preservation of
	this localization under the nonlinear evolution is assumed.
	
	The smallness required for the construction is measured by the
	critical norm
	\begin{align}\label{data:w small}
		\delta_J
		:=\left(
		\|w_0\|_{\dot H_D^{1/2}}^2
		+\|w_1\|_{\dot H_D^{-1/2}}^2
		\right)^{1/2}.
	\end{align}
	For $s>\frac12$, the support of the high-frequency multiplier gives
	\begin{align*}
		\delta_J
		\lesssim
		2^{-J(s-\frac12)}
		\left(
		\|u_0\|_{\dot H_D^s}^2
		+\|u_1\|_{\dot H_D^{s-1}}^2
		\right)^{1/2}
		\longrightarrow0
		\qquad\text{as }J\to\infty.
	\end{align*}
	At $s=\frac12$, the same conclusion follows from the convergence
	of the critical spectral tails:
	\begin{align*}
		\delta_J^2
		&\leq4\pi\int_{2^J}^{\infty}
		\left(
		\lambda|\mathcal F_Du_0(\lambda)|^2
		+\lambda^{-1}|\mathcal F_Du_1(\lambda)|^2
		\right)\,d\lambda
		\longrightarrow0.
	\end{align*}
	Consequently, for any prescribed sufficiently small $\epsilon>0$,
	we may fix $J$, depending on the initial data and $\epsilon$, such that
	$\delta_J\leq\epsilon$.
	
	We now apply the critical Strichartz estimate in
	Theorem~\ref{thm:St est}, obtained from
	\cite[Theorem~3.3]{XuYang:NLW}, to the Duhamel formulation of
	\eqref{eq:nlw}. A contraction argument in
	$L^4(\R\times\Omega)$ yields a global solution with small critical
	norm. The persistence of the initial $s$-level regularity is then
	established separately below.
	
	\subsection{Local well-posedness by Picard iteration}
	
	We use the negative-order spectral spaces defined in
	Section~\ref{sect:DF transf}.
	
	\begin{proposition}[Local well-posedness]\label{prop:general s local Picard}
		Let $\Omega=\R^3\backslash\overline{B(0,1)}$ and
		$\frac14\leq s\leq\frac34$.
		For every radial initial pair
		\begin{align*}
			(w_0,w_1)\in\dot H^s_{D,rad}(\Omega)
			\times\dot H^{s-1}_{D,rad}(\Omega),
		\end{align*}
		there exists $T>0$ such that the Dirichlet problem
		\begin{equation}\label{eq:general s local Picard}
			\begin{cases}
				\partial_t^2w-\Delta_\Omega w+w^3=0,&(t,x)\in[-T,T]\times\Omega,\\
				w(0,x)=w_0(x),\quad\partial_tw(0,x)=w_1(x),&x\in\Omega,
			\end{cases}
		\end{equation}
		has a unique mild solution satisfying
		\begin{align}\label{reg:general s local Picard}
			(w,\partial_tw)\in
			C([-T,T];\dot H^s_{D,rad}(\Omega)\times\dot H^{s-1}_{D,rad}(\Omega)).
		\end{align}
		Here $-\Delta_\Omega$ is the Dirichlet Laplacian, and the boundary condition
		is understood through its spectral propagators and Sobolev spaces.
		Writing
		\begin{align*}
			M=\|w_0\|_{\dot H_D^s}+\|w_1\|_{\dot H_D^{s-1}},
		\end{align*}
		the existence time may be chosen as $T=c_s(1+M)^{-2}$ for some $c_s>0$.
		The solution map is locally Lipschitz on a common local interval for data
		in a bounded subset of the initial-data space.
		If the maximal forward lifespan has a finite endpoint $T_+$, then
		\begin{align}\label{criterion:general s local Picard}
			\sup_{0\leq t<T_+}
			\left(\|w(t)\|_{\dot H_D^s}
			+\|\partial_tw(t)\|_{\dot H_D^{s-1}}\right)=\infty.
		\end{align}
		In particular, the proposition applies to the Fourier-truncated initial pair
		$(P^\Omega_{>2^J}u_0,P^\Omega_{>2^J}u_1)$.
	\end{proposition}
	
	\begin{proof}
		All functions and spaces in the proof are radial.
		We suppress the spatial domain in norms when there is no ambiguity.
		
		\noindent\textit{Step 1: The nonlinear estimates.}
		Since both $s$ and $1-s$ belong to $[\frac14,\frac34]$,
		Proposition~\ref{prop:radial Sob embedding} gives
		\begin{align*}
			\dot H^s_{D,rad}(\Omega)\hookrightarrow L^4(\Omega),
			\qquad \dot H^{1-s}_{D,rad}(\Omega)\hookrightarrow L^4(\Omega).
		\end{align*}
		The second embedding and duality imply
		\begin{align}\label{est:general s local source}
			\|F\|_{\dot H_D^{s-1}}\lesssim_s\|F\|_{L^{4/3}}
		\end{align}
		for radial $F\in L^{4/3}(\Omega)$.
		Indeed,
		\begin{align*}
			|\langle F,\varphi\rangle|
			\leq\|F\|_{L^{4/3}}\|\varphi\|_{L^4}
			\lesssim_s\|F\|_{L^{4/3}}\|\varphi\|_{\dot H_D^{1-s}}.
		\end{align*}
		Consequently, H\"older's inequality and
		$a^3-b^3=(a-b)(a^2+ab+b^2)$ yield
		\begin{align}\label{est:general s local cubic}
			\|a^3\|_{\dot H_D^{s-1}}&\lesssim_s\|a\|_{\dot H_D^s}^3,\\
			\|a^3-b^3\|_{\dot H_D^{s-1}}
			&\lesssim_s(\|a\|_{L^4}^2+\|b\|_{L^4}^2)\|a-b\|_{L^4}
			\nonumber\\
			&\lesssim_s(\|a\|_{\dot H_D^s}^2+\|b\|_{\dot H_D^s}^2)
			\|a-b\|_{\dot H_D^s}.
			\label{est:general s local cubic difference}
		\end{align}
		In particular, the cubic map is locally Lipschitz from
		$\dot H^s_{D,rad}$ to $\dot H^{s-1}_{D,rad}$.
		
		\noindent\textit{Step 2: Linear estimates and the iteration scheme.}
		The distorted Fourier multiplier identities give
		\begin{align*}
			\|\cos(t\sqrt{-\Delta_\Omega})F\|_{\dot H_D^s}&\leq\|F\|_{\dot H_D^s},\\
			\left\|\frac{\sin(t\sqrt{-\Delta_\Omega})}{\sqrt{-\Delta_\Omega}} F\right\|_{\dot H_D^s}
			&\leq\|F\|_{\dot H_D^{s-1}},\\
			\|\sqrt{-\Delta_\Omega}\sin(t\sqrt{-\Delta_\Omega})F\|_{\dot H_D^{s-1}}&\leq\|F\|_{\dot H_D^s},\\
			\|\cos(t\sqrt{-\Delta_\Omega})F\|_{\dot H_D^{s-1}}&\leq\|F\|_{\dot H_D^{s-1}}.
		\end{align*}
		For example, the second inequality follows from
		\begin{align*}
			\left\|\frac{\sin(t\sqrt{-\Delta_\Omega})}{\sqrt{-\Delta_\Omega}} F\right\|_{\dot H_D^s}^2
			&=4\pi\int_0^\infty\lambda^{2s-2}|\sin(t\lambda)|^2
			|\mathcal F_DF(\lambda)|^2\,d\lambda\\
			&\leq\|F\|_{\dot H_D^{s-1}}^2.
		\end{align*}
		These propagators are strongly continuous in the indicated spaces, as
		follows from dominated convergence in their spectral representations.
		
		Define
		\begin{align}\label{eq:general s local zero}
			w^{(0)}(t)=\cos(t\sqrt{-\Delta_\Omega})w_0+\frac{\sin(t\sqrt{-\Delta_\Omega})}{\sqrt{-\Delta_\Omega}} w_1.
		\end{align}
		For $n\geq0$, set
		\begin{align}\label{eq:general s local iteration}
			w^{(n+1)}(t)=w^{(0)}(t)
			-\int_0^t\frac{\sin((t-\tau)\sqrt{-\Delta_\Omega})}{\sqrt{-\Delta_\Omega}}
			(w^{(n)}(\tau))^3\,d\tau.
		\end{align}
		Whenever $w^{(n)}\in C_t\dot H_D^s$, the nonlinear estimates imply
		$(w^{(n)})^3\in C_t\dot H_D^{s-1}$.
		Thus the integral in \eqref{eq:general s local iteration} is a well-defined
		Banach-space integral in $\dot H_D^s$, and the next iterate is continuous
		in time in that space.
		
		\noindent\textit{Step 3: Uniform bounds on a common time interval.}
		Let $I_T=[-T,T]$ and
		\begin{align*}
			A_n=\|w^{(n)}\|_{C(I_T;\dot H_D^s)}.
		\end{align*}
		The linear bounds give $A_0\leq M$.
		For each $t\in I_T$, the triangle inequality, the spectral multiplier
		bound, and \eqref{est:general s local cubic} imply
		\begin{align*}
			\|w^{(n+1)}(t)\|_{\dot H_D^s}
			&\leq \|w^{(0)}(t)\|_{\dot H_D^s}\\
			&\quad+\left\|\int_0^t
			\frac{\sin((t-\tau)\sqrt{-\Delta_\Omega})}{\sqrt{-\Delta_\Omega}}
			(w^{(n)}(\tau))^3\,d\tau\right\|_{\dot H_D^s}\\
			&\leq M+\int_{\min\{0,t\}}^{\max\{0,t\}}
			\|(w^{(n)}(\tau))^3\|_{\dot H_D^{s-1}}\,d\tau\\
			&\leq M+C_s\int_{\min\{0,t\}}^{\max\{0,t\}}
			\|w^{(n)}(\tau)\|_{\dot H_D^s}^3\,d\tau\\
			&\leq M+C_s|t|A_n^3
			\leq M+C_sT A_n^3.
		\end{align*}
		Here the integration interval has length $|t|\leq T$, including when
		$t<0$. Taking the supremum over $t\in I_T$ gives
		\begin{align}\label{est:general s local induction}
			A_{n+1}\leq M+C_sT A_n^3.
		\end{align}
		Choose $R=2M+1$. Fix $K_s\geq1$, depending only on $s$,
		sufficiently large to dominate the constants in the nonlinear
		bounds and the estimates for successive differences below.
		After factoring out $R^2$ in the difference estimates,
		this choice is independent of $M$, $R$, and $T$.
		Take
		\begin{align}\label{choice:general s local time}
			T=\frac1{4K_sR^2}.
		\end{align}
		If $A_n\leq R$, then
		\begin{align*}
			A_{n+1}\leq M+K_sTR^3\leq\frac R2+\frac R4<R.
		\end{align*}
		Since $A_0\leq M<R$, induction gives
		\begin{align}\label{est:general s local uniform}
			\|w^{(n)}\|_{C(I_T;\dot H_D^s)}\leq R,
			\qquad n\geq0.
		\end{align}
		In particular, every iterate is defined on the same interval.
		
		\noindent\textit{Step 4: Convergence of successive differences.}
		For $n\geq1$, subtraction of two consecutive iteration formulas gives
		\begin{align*}
			w^{(n+1)}(t)-w^{(n)}(t)
			=-\int_0^t\frac{\sin((t-\tau)\sqrt{-\Delta_\Omega})}{\sqrt{-\Delta_\Omega}}
			\left((w^{(n)}(\tau))^3-(w^{(n-1)}(\tau))^3\right)d\tau.
		\end{align*}
		By \eqref{est:general s local cubic difference} and
		\eqref{est:general s local uniform},
		\begin{align*}
			\|w^{(n+1)}-w^{(n)}\|_{C(I_T;\dot H_D^s)}
			\leq K_sTR^2\|w^{(n)}-w^{(n-1)}\|_{C(I_T;\dot H_D^s)}
			=\frac14\|w^{(n)}-w^{(n-1)}\|_{C(I_T;\dot H_D^s)}.
		\end{align*}
		Therefore,
		\begin{align}\label{est:general s local geometric}
			\|w^{(n+1)}-w^{(n)}\|_{C(I_T;\dot H_D^s)}
			\leq4^{-n}\|w^{(1)}-w^{(0)}\|_{C(I_T;\dot H_D^s)}.
		\end{align}
		The series of successive differences converges in the complete space
		$C(I_T;\dot H^s_{D,rad})$.
		Hence there exists $w\in C(I_T;\dot H^s_{D,rad})$ such that
		\begin{align}\label{conv:general s local}
			w^{(n)}\longrightarrow w\quad\hbox{in }C(I_T;\dot H_D^s),
			\qquad \|w\|_{C(I_T;\dot H_D^s)}\leq R.
		\end{align}
		Equivalently, the Duhamel map is a contraction on the closed ball
		$\{w\in C(I_T;\dot H^s_{D,rad}):\|w\|_{C_t\dot H_D^s}\leq R\}$.
		
		\noindent\textit{Step 5: Time derivatives and passage to the limit.}
		The time derivative of each iterate is
		\begin{align}\label{eq:general s local derivative}
			\partial_t w^{(n+1)}(t)
			={}&-\sqrt{-\Delta_\Omega}
			\sin(t\sqrt{-\Delta_\Omega})w_0
			+\cos(t\sqrt{-\Delta_\Omega})w_1
			\nonumber\\
			&-\int_0^t
			\cos((t-\tau)\sqrt{-\Delta_\Omega})
			(w^{(n)}(\tau))^3\,d\tau.
		\end{align}
		The formula follows first for smooth spectral data and then by the
		linear estimates; time derivatives are understood distributionally.
		The endpoint term from differentiating the integral vanishes because
		$\left.\frac{\sin(t\sqrt{-\Delta_\Omega})}{\sqrt{-\Delta_\Omega}}\right|_{t=0}=0$.
		For $n\geq1$, the same nonlinear difference estimate yields
		\begin{align*}
			&\|\partial_tw^{(n+1)}-\partial_tw^{(n)}\|_{C(I_T;\dot H_D^{s-1})}\\
			&\qquad\leq K_sTR^2
			\|w^{(n)}-w^{(n-1)}\|_{C(I_T;\dot H_D^s)}.
		\end{align*}
		The right side is summable by \eqref{est:general s local geometric}.
		Thus the derivatives converge in $C(I_T;\dot H_D^{s-1})$, and their limit
		is the distributional time derivative of the limit in
		\eqref{conv:general s local}.
		Moreover,
		\begin{align*}
			\|(w^{(n)})^3-w^3\|_{C(I_T;\dot H_D^{s-1})}
			\lesssim_s R^2\|w^{(n)}-w\|_{C(I_T;\dot H_D^s)}\longrightarrow0.
		\end{align*}
		Passing to the limit in \eqref{eq:general s local iteration} gives
		\begin{align}\label{eq:general s local mild}
			w(t)=\cos(t\sqrt{-\Delta_\Omega})w_0+\frac{\sin(t\sqrt{-\Delta_\Omega})}{\sqrt{-\Delta_\Omega}} w_1
			-\int_0^t\frac{\sin((t-\tau)\sqrt{-\Delta_\Omega})}{\sqrt{-\Delta_\Omega}} w(\tau)^3\,d\tau.
		\end{align}
		Passing to the limit in \eqref{eq:general s local derivative} also gives
		the corresponding derivative formula, so
		$w(0)=w_0$ and $\partial_tw(0)=w_1$.
		This proves existence and \eqref{reg:general s local Picard}.
		
		\noindent\textit{Step 6: Uniqueness and continuous dependence.}
		For two solutions in the ball above, the difference estimate and
		\eqref{choice:general s local time} imply
		\begin{align*}
			\|w-\widetilde w\|_{C_t\dot H_D^s}
			&\leq\|w_0-\widetilde w_0\|_{\dot H_D^s}
			+\|w_1-\widetilde w_1\|_{\dot H_D^{s-1}}
			+\frac14\|w-\widetilde w\|_{C_t\dot H_D^s}.
		\end{align*}
		After absorption, and using the derivative formula for the difference,
		we obtain
		\begin{align}\label{est:general s local stability}
			&\|w-\widetilde w\|_{C_t\dot H_D^s}
			+\|\partial_tw-\partial_t\widetilde w\|_{C_t\dot H_D^{s-1}}
			\nonumber\\
			&\qquad\lesssim_s\|w_0-\widetilde w_0\|_{\dot H_D^s}
			+\|w_1-\widetilde w_1\|_{\dot H_D^{s-1}}.
		\end{align}
		For data in a fixed bounded set, choose a common radius and time interval
		as above. This proves local Lipschitz dependence.
		In particular, equal initial data give equal solutions in the ball.
		For arbitrary continuous mild solutions on a common interval, their norms
		are bounded on compact subintervals. Subdividing into sufficiently short
		intervals and restarting the difference estimate proves uniqueness in
		the entire class \eqref{reg:general s local Picard}.
		
		The time in \eqref{choice:general s local time} is bounded below by
		$c_s(1+M)^{-2}$; restriction to this smaller time gives the stated choice.
		Restarting the same construction at any time proves the continuation
		criterion \eqref{criterion:general s local Picard}.
		Finally, $P^\Omega_{>2^J}$ is bounded on each spectral Sobolev space, so
		the conclusion applies to the Fourier-truncated initial data as well.
	\end{proof}
	
	\subsection{Global well-posedness of the high-frequency part}
	
	\begin{proposition}[Global well-posedness of the high-frequency part]
		\label{prop:high frequency global revised}
		Let $\frac12\leq s\leq\frac34$ and let the radial initial pair satisfy
		\begin{align*}
			(u_0,u_1)\in\dot H^s_{D,rad}(\Omega)
			\times\dot H^{s-1}_{D,rad}(\Omega).
		\end{align*}
		For a sufficiently large integer $J$, set
		\begin{align}\label{data:high frequency global revised}
			w_0=P^\Omega_{>2^J}u_0,\qquad w_1=P^\Omega_{>2^J}u_1.
		\end{align}
		Then the Dirichlet problem
		\begin{equation}\label{eq:high frequency global revised}
			\begin{cases}
				\partial_t^2w-\Delta_\Omega w+w^3=0,&(t,x)\in\R\times\Omega,\\
				w(0,x)=w_0(x),\quad\partial_tw(0,x)=w_1(x),&x\in\Omega,
			\end{cases}
		\end{equation}
		has a unique global mild solution satisfying
		\begin{align}\label{reg:high frequency global revised}
			(w,\partial_tw)\in C(\R;
			\dot H^s_{D,rad}(\Omega)\times\dot H^{s-1}_{D,rad}(\Omega)).
		\end{align}
		The boundary condition is understood through the Dirichlet spectral
		propagators and Sobolev spaces.
		Define
		\begin{align*}
			\delta_J=\left(\|w_0\|_{\dot H_D^{1/2}}^2
			+\|w_1\|_{\dot H_D^{-1/2}}^2\right)^{1/2},\qquad
			M_{s,J}=\|w_0\|_{\dot H_D^s}+\|w_1\|_{\dot H_D^{s-1}}.
		\end{align*}
		There exists an absolute threshold $\epsilon_0>0$ such that $J$ may be
		chosen to ensure $\delta_J\leq\epsilon_0$, and the resulting solution obeys
		\begin{align}\label{est:high frequency critical revised}
			&\|w\|_{L^4_{t,x}(\R\times\Omega)}
			+\sup_{t\in\R}\left(
			\|w(t)\|_{\dot H_D^{1/2}}
			+\|\partial_tw(t)\|_{\dot H_D^{-1/2}}\right)
			\lesssim\delta_J.
		\end{align}
		For every $T>0$, it also satisfies
		\begin{align}\label{est:high frequency s revised}
			\sup_{|t|\leq T}\left(
			\|w(t)\|_{\dot H_D^s}
			+\|\partial_tw(t)\|_{\dot H_D^{s-1}}\right)
			\lesssim_s M_{s,J}+T\delta_J^3.
		\end{align}
		The solution depends locally Lipschitz continuously on the initial pair
		in the $s$-level topology on each finite time interval, within this
		small-critical-data class.
	\end{proposition}
	\begin{proof}
		All functions in the proof are radial. We suppress the spatial domain in norms.
		
		\noindent\textit{Step 1: Smallness of the truncated critical data.}
		The definition of the distorted Fourier projections gives
		\begin{align*}
			\mathcal F_Dw_j(\lambda)
			=\bigl(1-\phi(\lambda/2^J)\bigr)\mathcal F_Du_j(\lambda),
			\qquad j=0,1.
		\end{align*}
		The multiplier vanishes for $0\leq\lambda\leq2^J$ and has absolute value
		at most one. If $s>\frac12$, then, for $\lambda>2^J$,
		\begin{align*}
			\lambda&\leq2^{J(1-2s)}\lambda^{2s},\\
			\lambda^{-1}&\leq2^{J(1-2s)}\lambda^{2(s-1)}.
		\end{align*}
		Consequently, the spectral norm identities imply
		\begin{align}\label{est:high frequency tail rate revised}
			\delta_J\leq2^{J(1/2-s)}
			\left(\|u_0\|_{\dot H_D^s}+\|u_1\|_{\dot H_D^{s-1}}\right)
			\longrightarrow0\qquad\text{as }J\to\infty.
		\end{align}
		At $s=\frac12$, instead use
		\begin{align*}
			\|w_0\|_{\dot H_D^{1/2}}^2+\|w_1\|_{\dot H_D^{-1/2}}^2
			\leq4\pi\int_{2^J}^\infty
			\left(\lambda|\mathcal F_Du_0(\lambda)|^2
			+\lambda^{-1}|\mathcal F_Du_1(\lambda)|^2\right)d\lambda
			\longrightarrow0.
		\end{align*}
		Thus $\delta_J$ can be made arbitrarily small throughout the stated range.
		
		\noindent\textit{Step 2: Global critical Picard iteration.}
		Theorem~\ref{thm:St est} gives the following estimate
		for the radial linear equation $z_{tt}-\Delta_\Omega z=F$:
		\begin{align}\label{est:high frequency critical linear revised}
			&\|z\|_{L^4_{t,x}(\R\times\Omega)}
			+\|z\|_{C_b(\R;\dot H_D^{1/2})}
			+\|\partial_tz\|_{C_b(\R;\dot H_D^{-1/2})}\nonumber\\
			&\qquad\lesssim
			\|z(0)\|_{\dot H_D^{1/2}}+\|\partial_tz(0)\|_{\dot H_D^{-1/2}}
			+\|F\|_{L^{4/3}_{t,x}(\R\times\Omega)}.
		\end{align}
		Here $C_b$ denotes the space of bounded continuous functions
		with values in the indicated Sobolev space.
		The estimate already includes the Besov-to-Lebesgue embeddings
		established in the proof of Theorem~\ref{thm:St est}.
		
		Define the free solution
		\begin{align}\label{eq:high frequency free revised}
			w^{(0)}(t)=\cos(t\sqrt{-\Delta_\Omega})w_0
			+\frac{\sin(t\sqrt{-\Delta_\Omega})}{\sqrt{-\Delta_\Omega}}w_1
		\end{align}
		and, for $n\geq0$, define
		\begin{align}\label{eq:high frequency iteration revised}
			w^{(n+1)}(t)=w^{(0)}(t)
			-\int_0^t
			\frac{\sin((t-\tau)\sqrt{-\Delta_\Omega})}{\sqrt{-\Delta_\Omega}}
			(w^{(n)}(\tau))^3\,d\tau.
		\end{align}
		We perform the iteration in $L^4(\R\times\Omega)$ alone. Set
		\begin{align*}
			A_n=\|w^{(n)}\|_{L^4_{t,x}(\R\times\Omega)}.
		\end{align*}
		By \eqref{est:high frequency critical linear revised},
		\begin{align*}
			A_0\leq C\delta_J,\qquad
			A_{n+1}\leq C\delta_J+C\|(w^{(n)})^3\|_{L^{4/3}_{t,x}}
			=C\delta_J+CA_n^3.
		\end{align*}
		Choose $C\geq1$ and set $R=2C\delta_J$.
		For sufficiently small $\epsilon_0$, induction gives
		\begin{align}\label{est:high frequency iterate bound revised}
			A_n\leq R,\qquad n\geq0.
		\end{align}
		Indeed, $A_n\leq R$ implies $A_{n+1}\leq R/2+CR^3\leq R$
		provided that $CR^2\leq1/2$.
		
		By H\"older's inequality,
		\begin{align*}
			\|a^3-b^3\|_{L^{4/3}_{t,x}}
			\lesssim(\|a\|_{L^4_{t,x}}^2+\|b\|_{L^4_{t,x}}^2)
			\|a-b\|_{L^4_{t,x}}.
		\end{align*}
		Subtracting consecutive iteration formulas and applying
		\eqref{est:high frequency critical linear revised} therefore gives
		\begin{align*}
			\|w^{(n+1)}-w^{(n)}\|_{L^4_{t,x}}
			\leq KR^2\|w^{(n)}-w^{(n-1)}\|_{L^4_{t,x}},\qquad n\geq1,
		\end{align*}
		where $K$ is independent of $J$ and the data.
		Decrease $\epsilon_0$ so that $KR^2\leq1/2$.
		The successive differences form a summable geometric series in the complete
		space $L^4(\R\times\Omega)$. Hence $w^{(n)}\to w$ in that space, with
		$\|w\|_{L^4_{t,x}}\leq R$.
		The remaining parts of the same Strichartz estimate give
		\begin{align*}
			&\|w^{(n+1)}-w^{(n)}\|_{C_b\dot H_D^{1/2}}
			+\|\partial_tw^{(n+1)}-\partial_tw^{(n)}\|_{C_b\dot H_D^{-1/2}}\\
			&\qquad\lesssim R^2\|w^{(n)}-w^{(n-1)}\|_{L^4_{t,x}}.
		\end{align*}
		Since the right side is summable, the iterates also converge in
		$C_b(\R;\dot H_D^{1/2})$, and their derivatives converge in
		$C_b(\R;\dot H_D^{-1/2})$.
		The limits agree with the $L^4_{t,x}$ limit as distributions, and the
		limit of the derivatives is the distributional time derivative of $w$.
		Moreover, $(w^{(n)})^3\to w^3$ in $L^{4/3}_{t,x}$.
		Passing to the limit gives
		\begin{align}\label{eq:high frequency mild revised}
			w(t)=w^{(0)}(t)
			-\int_0^t
			\frac{\sin((t-\tau)\sqrt{-\Delta_\Omega})}{\sqrt{-\Delta_\Omega}}
			w(\tau)^3\,d\tau.
		\end{align}
		The differentiated formula similarly passes to the limit and recovers
		$\partial_tw(0)=w_1$; also $w(0)=w_0$.
		Applying \eqref{est:high frequency critical linear revised} to the limit gives
		\begin{align}\label{est:high frequency recovered critical revised}
			&\|w\|_{C_b\dot H_D^{1/2}}
			+\|\partial_tw\|_{C_b\dot H_D^{-1/2}}\nonumber\\
			&\qquad\lesssim\delta_J+\|w^3\|_{L^{4/3}_{t,x}}
			=\delta_J+\|w\|_{L^4_{t,x}}^3
			\lesssim\delta_J+\delta_J^3\lesssim\delta_J,
		\end{align}
		where $\epsilon_0\leq1$ is assumed. Together with
		$\|w\|_{L^4_{t,x}}\leq2C\delta_J$, this proves
		\eqref{est:high frequency critical revised}.
		
		\noindent\textit{Step 3: Picard convergence at the $s$ level.}
		
		We prove convergence of the same Picard sequence in
		$C([-T_s,T_s];\dot H_D^s)$.
		
		First, the full linear estimate
		\eqref{est:high frequency critical linear revised}, applied to every
		iterate, and \eqref{est:high frequency iterate bound revised} imply
		\begin{align}\label{est:high frequency iterates critical spatial}
			\sup_{n\geq0}\sup_{t\in\R}
			\|w^{(n)}(t)\|_{\dot H_D^{1/2}}
			\lesssim\delta_J+R^3\lesssim\delta_J.
		\end{align}
		For $n=0$ this follows from the homogeneous estimate; for $n\geq1$
		use the source $(w^{(n-1)})^3$.
		By the radial embedding $\dot H_D^{1/2}\hookrightarrow L^4$, we have
		\begin{align}\label{est:high frequency iterates L4 spatial}
			\sup_{n\geq0}\sup_{t\in\R}\|w^{(n)}(t)\|_{L^4}
			\lesssim\delta_J.
		\end{align}
		This is a uniform-in-time spatial estimate; it is obtained from the
		critical Sobolev estimate, not from the spacetime $L^4$ bound alone.
		
		Since $\frac12\leq s\leq\frac34$, both $s$ and $1-s$ lie in
		$[\frac14,\frac34]$. Proposition~\ref{prop:radial Sob embedding} and
		its dual therefore give
		\begin{align}\label{est:high frequency source direct}
			\|a\|_{L^4}\lesssim_s\|a\|_{\dot H_D^s},
			\qquad
			\|F\|_{\dot H_D^{s-1}}\lesssim_s\|F\|_{L^{4/3}}
		\end{align}
		for radial functions. In particular, whenever $a,b$ also satisfy
		$\|a\|_{L^4}+\|b\|_{L^4}\leq C\delta_J$, H\"older's inequality gives
		\begin{align}\label{est:high frequency s cubic iteration}
			\|a^3\|_{\dot H_D^{s-1}}
			&\lesssim_s\|a\|_{L^4}^3
			\lesssim_s\delta_J^2\|a\|_{\dot H_D^s},\\
			\|a^3-b^3\|_{\dot H_D^{s-1}}
			&\lesssim_s(\|a\|_{L^4}^2+\|b\|_{L^4}^2)\|a-b\|_{L^4}
			\nonumber\\
			&\lesssim_s\delta_J^2\|a-b\|_{\dot H_D^s}.
			\label{est:high frequency s difference iteration}
		\end{align}
		The constants here include the fixed constant from
		\eqref{est:high frequency iterates L4 spatial}.
		
		The source is measured in $\dot H_D^{s-1}$ because the sine propagator
		has the spectral multiplier $\sin(t\lambda)/\lambda$:
		\begin{align}\label{est:high frequency s multipliers}
			\left\|\frac{\sin(t\sqrt{-\Delta_\Omega})}{\sqrt{-\Delta_\Omega}}F
			\right\|_{\dot H_D^s}
			&\leq\|F\|_{\dot H_D^{s-1}},\\
			\|\cos(t\sqrt{-\Delta_\Omega})F\|_{\dot H_D^{s-1}}
			&\leq\|F\|_{\dot H_D^{s-1}}.\nonumber
		\end{align}
		Thus, for the spectral energy estimates used here,
		it suffices to control the source in
		$L^1(I_s;\dot H_D^{s-1})$
		to obtain displacement in $C(I_s;\dot H_D^s)$
		and velocity in $C(I_s;\dot H_D^{s-1})$.
		
		\smallskip
		\noindent\textit{Uniform $s$-level bounds.}
		Let $I_s=[-T_s,T_s]$ and set
		\begin{align*}
			B_n=\|w^{(n)}\|_{C(I_s;\dot H_D^s)}.
		\end{align*}
		The free solution satisfies $B_0\leq M_{s,J}$.
		Inductively, if $w^{(n)}\in C(I_s;\dot H_D^s)$, its cubic source is
		continuous in $\dot H_D^{s-1}$ by the preceding estimates.
		Hence the next iterate belongs to the same displacement space, and
		\begin{align}\label{est:high frequency s induction}
			B_{n+1}
			&\leq M_{s,J}
			+\sup_{|t|\leq T_s}\int_{\min\{0,t\}}^{\max\{0,t\}}
			\|(w^{(n)}(\tau))^3\|_{\dot H_D^{s-1}}\,d\tau
			\nonumber\\
			&\leq M_{s,J}+K_sT_s\delta_J^2 B_n.
		\end{align}
		Choose $K_s\geq1$, depending only on $s$, large enough to dominate
		this constant and the difference constants below, and take
		\begin{align}\label{choice:high frequency s iteration time}
			T_s=\frac{1}{4K_s(1+\delta_J^2)}.
		\end{align}
		Then $K_sT_s\delta_J^2\leq1/4$.
		If $B_n\leq2M_{s,J}$, \eqref{est:high frequency s induction} gives
		$B_{n+1}\leq M_{s,J}+\frac14(2M_{s,J})\leq2M_{s,J}$.
		Since $B_0\leq M_{s,J}$, induction yields
		\begin{align}\label{est:high frequency s iterates uniform}
			B_n\leq2M_{s,J},\qquad n\geq0.
		\end{align}
		
		\smallskip
		\noindent\textit{Contraction in the $s$-level displacement norm.}
		Set
		\begin{align*}
			d_n=\|w^{(n+1)}-w^{(n)}\|_{C(I_s;\dot H_D^s)}.
		\end{align*}
		For $n\geq1$, consecutive iterates have the same initial data, so their
		homogeneous terms cancel. By
		\eqref{est:high frequency s difference iteration} and
		\eqref{est:high frequency s multipliers},
		\begin{align}\label{est:high frequency s contraction}
			d_n
			&\leq\sup_{|t|\leq T_s}\int_{\min\{0,t\}}^{\max\{0,t\}}
			\|(w^{(n)}(\tau))^3-(w^{(n-1)}(\tau))^3\|_{\dot H_D^{s-1}}\,d\tau
			\nonumber\\
			&\leq K_sT_s\delta_J^2d_{n-1}\leq\frac14d_{n-1}.
		\end{align}
		Consequently $d_n\leq4^{-n}d_0$, and the same Picard sequence converges
		in $C(I_s;\dot H_D^s)$.
		This contraction estimate is applied to the sequence already controlled
		by \eqref{est:high frequency iterates L4 spatial}; no critical bound is
		being inferred from an arbitrary $s$-level ball.
		
		For completeness, the derivatives of the iterates satisfy
		\begin{align*}
			\partial_tw^{(n+1)}(t)
			={}&-\sqrt{-\Delta_\Omega}\sin(t\sqrt{-\Delta_\Omega})w_0
			+\cos(t\sqrt{-\Delta_\Omega})w_1\\
			&-\int_0^t\cos((t-\tau)\sqrt{-\Delta_\Omega})
			(w^{(n)}(\tau))^3\,d\tau.
		\end{align*}
		Subtracting these formulas gives, for $n\geq1$,
		\begin{align*}
			\|\partial_tw^{(n+1)}-\partial_tw^{(n)}\|_{C(I_s;\dot H_D^{s-1})}
			\leq K_sT_s\delta_J^2d_{n-1}.
		\end{align*}
		The right side is summable. Therefore the pair converges in
		$C(I_s;\dot H_D^s\times\dot H_D^{s-1})$.
		The limits agree with the critical limit of Step 2 as distributions.
		The nonlinear difference estimate also permits passage to the limit
		in the $s$-level Duhamel formulas. Thus the global critical solution
		has the stated $s$-level regularity at least on $I_s$.
		
		\noindent\textit{Step 4: Global continuation at the $s$ level.}
		Extend the preceding $s$-level solution to its maximal interval using
		Proposition~\ref{prop:general s local Picard}.
		On each compact subinterval, it also belongs to the continuous critical
		class: its initial data are critical, and
		$w\in C_t\dot H_D^s\hookrightarrow C_tL^4$ implies
		$w^3\in C_tL^{4/3}\hookrightarrow C_t\dot H_D^{-1/2}$.
		The spectral Duhamel estimates then give the critical regularity there.
		Uniqueness in the critical local theory identifies this solution with
		the global critical solution constructed in Step 2.
		
		The critical bound and Proposition~\ref{prop:critical sqrt p} with $p=4,6$
		therefore imply
		\begin{align}\label{est:high frequency spatial revised}
			\sup_{t\in\R}\bigl(\|w(t)\|_{L^4}+\|w(t)\|_{L^6}\bigr)
			&\leq C(2+\sqrt6)\sup_{t\in\R}\|w(t)\|_{\dot H_D^{1/2}}
			\nonumber\\
			&\lesssim\delta_J.
		\end{align}
		Only the $L^4$ part is needed for continuation in the present range.
		Indeed, \eqref{est:high frequency source direct} gives
		\begin{align}\label{est:high frequency cubic source revised}
			\|w(t)^3\|_{\dot H_D^{s-1}}
			\lesssim_s\|w(t)^3\|_{L^{4/3}}
			=\|w(t)\|_{L^4}^3\lesssim_s\delta_J^3.
		\end{align}
		Using the $s$-level Duhamel formula and its time derivative, we obtain,
		throughout the maximal $s$-level interval,
		\begin{align}\label{est:high frequency s continuation}
			\|w(t)\|_{\dot H_D^s}+\|\partial_tw(t)\|_{\dot H_D^{s-1}}
			&\lesssim M_{s,J}
			+\int_{\min\{0,t\}}^{\max\{0,t\}}
			\|w(\tau)^3\|_{\dot H_D^{s-1}}\,d\tau
			\nonumber\\
			&\lesssim_s M_{s,J}+|t|\delta_J^3.
		\end{align}
		If a maximal endpoint were finite, this bound would keep the full
		$s$-level norm bounded up to that endpoint, contradicting the continuation
		criterion in Proposition~\ref{prop:general s local Picard}.
		Both endpoints are therefore infinite. This proves
		\eqref{reg:high frequency global revised} and
		\eqref{est:high frequency s revised}.
		
		\noindent\textit{Step 5: Uniqueness and continuous dependence.}
		We give the finite-interval argument explicitly. Fix $T>0$.
		For a pair $(a_0,a_1)$, write its $s$-level size as
		$\|a_0\|_{\dot H_D^s}+\|a_1\|_{\dot H_D^{s-1}}$.
		
		\smallskip
		\noindent\textit{Uniqueness in the continuous $s$-level mild class.}
		Let $w$ and $\widetilde w$ be two mild solutions with the same initial
		pair and with continuous $s$-level regularity on $[-T,T]$.
		Their continuity implies that there is a finite number $B\geq1$ such that
		\begin{align*}
			\sup_{|t|\leq T}
			\bigl(\|w(t)\|_{\dot H_D^s}
			+\|\widetilde w(t)\|_{\dot H_D^s}\bigr)\leq B.
		\end{align*}
		No small-critical-norm assumption is needed for this comparison.
		The nonlinear estimate from Proposition~\ref{prop:general s local Picard}
		therefore gives, for $|t|\leq T$,
		\begin{align}\label{est:high frequency uniqueness source detailed}
			\|w(t)^3-\widetilde w(t)^3\|_{\dot H_D^{s-1}}
			\leq C_sB^2\|w(t)-\widetilde w(t)\|_{\dot H_D^s}.
		\end{align}
		Choose $h>0$ with $h\leq1$ and $C_s hB^2\leq1/2$.
		On $[0,\min\{h,T\}]$, the homogeneous terms cancel, so subtraction of
		both Duhamel formulas and the sine multiplier bound yield
		\begin{align*}
			\|w-\widetilde w\|_{C([0,\min\{h,T\}];\dot H_D^s)}
			\leq C_s hB^2
			\|w-\widetilde w\|_{C([0,\min\{h,T\}];\dot H_D^s)}.
		\end{align*}
		Absorption gives equality of the displacements on this interval.
		The differentiated Duhamel formula then gives equality of the velocities.
		If $h<T$, the initial pairs at time $h$ consequently agree.
		Restart the same argument at $h$ and continue over intervals of length at
		most $h$. Since the bound $B$ is valid on all of $[0,T]$, the same $h$
		works at every step, and finitely many steps cover $[0,T]$.
		The identical argument with time reversed covers $[-T,0]$.
		Thus $w=\widetilde w$ and $\partial_tw=\partial_t\widetilde w$ on $[-T,T]$.
		Since $T$ is arbitrary, uniqueness holds globally in the stated mild class.
		
		\smallskip
		\noindent\textit{A common bound for nearby small-critical data.}
		Now allow the initial pairs to differ, and assume that
		\begin{align*}
			M_{s,J},\ \widetilde M_{s,J}\leq M_*,\qquad
			\delta_J,\ \widetilde\delta_J\leq\epsilon_0.
		\end{align*}
		The preceding construction supplies both global solutions.
		By \eqref{est:high frequency s revised}, there is a common bound
		\begin{align}\label{est:high frequency common bound detailed}
			&\sup_{|t|\leq T}\bigl(
			\|w(t)\|_{\dot H_D^s}+\|\partial_tw(t)\|_{\dot H_D^{s-1}}
			+\|\widetilde w(t)\|_{\dot H_D^s}
			+\|\partial_t\widetilde w(t)\|_{\dot H_D^{s-1}}\bigr)
			\leq B_T,\nonumber\\
			&\hspace{25mm} B_T=1+C_s(M_*+T\epsilon_0^3),
		\end{align}
		where $C_s$ is enlarged if necessary.
		
		\smallskip
		\noindent\textit{Lipschitz control on a common time partition.}
		Choose a constant $L_s\geq1$ large enough for the following pairwise
		Duhamel estimate, and let
		\begin{align*}
			h_T=\min\left\{1,\frac{1}{2L_sB_T^2}\right\}.
		\end{align*}
		Partition $[0,T]$ as
		$0=t_0<t_1<\cdots<t_N=T$, with $t_{k+1}-t_k\leq h_T$.
		The integer $N$ can be chosen in terms of $s,T,M_*$ only.
		For $I_k=[t_k,t_{k+1}]$, define
		\begin{align*}
			Q_k={}&\|w(t_k)-\widetilde w(t_k)\|_{\dot H_D^s}
			+\|\partial_tw(t_k)-\partial_t\widetilde w(t_k)\|_{\dot H_D^{s-1}},\\
			E_k={}&\|w-\widetilde w\|_{C(I_k;\dot H_D^s)}
			+\|\partial_tw-\partial_t\widetilde w\|_{C(I_k;\dot H_D^{s-1})}.
		\end{align*}
		Subtract the two Duhamel formulas with initial time $t_k$, and also
		subtract their time derivatives. The homogeneous displacement and
		velocity terms together are bounded by $2Q_k$.
		The nonlinear terms are bounded by
		$L_s(t_{k+1}-t_k)B_T^2E_k$, using the cubic difference estimate and
		\eqref{est:high frequency common bound detailed}.
		Consequently,
		\begin{align*}
			E_k\leq2Q_k+L_sh_TB_T^2E_k
			\leq2Q_k+\frac12E_k,
		\end{align*}
		so that
		\begin{align*}
			E_k\leq4Q_k,\qquad Q_{k+1}\leq E_k\leq4Q_k.
		\end{align*}
		Induction gives $Q_k\leq4^kQ_0$ and $E_k\leq4^{k+1}Q_0$.
		Applying the same argument backwards from zero yields
		\begin{align}\label{est:high frequency finite time stability detailed}
			&\|w-\widetilde w\|_{C([-T,T];\dot H_D^s)}
			+\|\partial_tw-\partial_t\widetilde w\|_{C([-T,T];\dot H_D^{s-1})}
			\nonumber\\
			&\qquad\leq C_{s,T,M_*}
			\bigl(\|w_0-\widetilde w_0\|_{\dot H_D^s}
			+\|w_1-\widetilde w_1\|_{\dot H_D^{s-1}}\bigr).
		\end{align}
		The constant is finite for every finite $T$.
		This proves Lipschitz dependence on each finite interval in the
		small-critical-data class with bounded $s$-level size.
		
		\smallskip
		\noindent\textit{Fixing the cutoff near the original initial pair.}
		For a fixed original pair $(u_0,u_1)$, Step 1 allows us to choose $J$
		so that $\delta_J\leq\epsilon_0/2$.
		Keep this $J$ fixed. For a nearby original pair
		$(\widetilde u_0,\widetilde u_1)$, use the same projections to define
		$(\widetilde w_0,\widetilde w_1)$.
		The spectral support estimate gives
		\begin{align*}
			&\|P^\Omega_{>2^J}(u_0-\widetilde u_0)\|_{\dot H_D^{1/2}}
			+\|P^\Omega_{>2^J}(u_1-\widetilde u_1)\|_{\dot H_D^{-1/2}}\\
			&\qquad\leq2^{J(1/2-s)}
			\bigl(\|u_0-\widetilde u_0\|_{\dot H_D^s}
			+\|u_1-\widetilde u_1\|_{\dot H_D^{s-1}}\bigr).
		\end{align*}
		This includes $s=1/2$, where the factor equals one.
		Thus a sufficiently small neighborhood of $(u_0,u_1)$ ensures that
		$\widetilde\delta_J\leq\epsilon_0$ for the same fixed $J$.
		The truncated $s$-level norms are bounded on this neighborhood, since
		$|1-\phi(\lambda/2^J)|\leq1$.
		The same bound also implies
		\begin{align*}
			\|w_0-\widetilde w_0\|_{\dot H_D^s}
			+\|w_1-\widetilde w_1\|_{\dot H_D^{s-1}}
			\leq\|u_0-\widetilde u_0\|_{\dot H_D^s}
			+\|u_1-\widetilde u_1\|_{\dot H_D^{s-1}}.
		\end{align*}
		Combining this with
		\eqref{est:high frequency finite time stability detailed} proves local
		Lipschitz dependence of the high-frequency solution on the original
		initial pair, with a fixed cutoff in that neighborhood.
		This completes the proof for the nonlinear high-frequency component.
	\end{proof}
	
	\begin{remark}
		The argument proves the high-frequency result for
		$\frac12\leq s\leq\frac34$.
		The local theory at $\frac14\leq s<\frac12$ does not supply the critical
		smallness required in Step 2 for general data.
		Although the initial pair is high-frequency localized, the nonlinear
		solution is not asserted to retain that spectral localization.
	\end{remark}
	
	\subsection{Global energy analysis for the low-frequency remainder}
	
	Let $\frac12\leq s\leq\frac34$ and let
	$(u_0,u_1)\in\dot H^s_{D,rad}(\Omega)\times\dot H^{s-1}_{D,rad}(\Omega)$.
	Fix $J$ as in Proposition~\ref{prop:high frequency global revised}, and
	let $w$ be the global solution with initial pair
	$(P^\Omega_{>2^J}u_0,P^\Omega_{>2^J}u_1)$.
	Set
	\begin{align}\label{data:low frequency revised}
		v_0=P^\Omega_{\leq2^J}u_0,\qquad
		v_1=P^\Omega_{\leq2^J}u_1.
	\end{align}
	The remainder equation associated with $u=v+w$ is
	\begin{equation}\label{eq:low frequency revised}
		\begin{cases}
			\partial_t^2v-\Delta_\Omega v+v^3=-3v^2w-3vw^2,
			&(t,x)\in\R\times\Omega,\\
			v(0)=v_0,\qquad\partial_tv(0)=v_1.
		\end{cases}
	\end{equation}
	There is no $w^3$ term in this equation, since $w$ solves
	$\partial_t^2w-\Delta_\Omega w+w^3=0$.
	Only the initial pair in \eqref{data:low frequency revised} is
	low-frequency localized; the solution $v$ is not assumed to retain
	this spectral localization.
	
	We use the nonhomogeneous space $H_D^1(\Omega)$ already defined in the
	Sobolev-space section. At this index it agrees with the usual zero-trace space:
	\begin{align}\label{def:low frequency H01 detailed}
		H_D^1(\Omega)=H_0^1(\Omega)
		=\overline{C_c^\infty(\Omega)}^{\,H^1(\Omega)},\qquad
		\|h\|_{H_D^1}^2=\|h\|_2^2+\|\nabla h\|_2^2.
	\end{align}
	Indeed, on $C_c^\infty(\Omega)$, the spectral norm
	$\|(I-\Delta_\Omega)^{1/2}h\|_2$ has exactly this squared value,
	and both spaces are obtained by completion.
	The homogeneous energy space controls only $\|\nabla v\|_2$.
	We use $H_D^1$ for the increment $h=v-v_0$; we require neither
	$v_0\in L^2$ nor $v\in L^2$.
	The physical energy remains the homogeneous energy $E(v)$ below;
	no $L^2$ term is added to that energy.
	
	The high-frequency result provides
	$w\in C_b(\R;\dot H_D^{1/2})$.
	The quantitative radial embedding consequently implies
	\begin{align}\label{est:low frequency background properties revised}
		w\in C(\R;L^p(\Omega)),\qquad
		\sup_{t\in\R}\|w(t)\|_p\leq C\sqrt p\,B,
		\quad 4\leq p<\infty,
	\end{align}
	where $B=\sup_t\|w(t)\|_{\dot H_D^{1/2}}$.
	Continuity follows by applying the same embedding to $w(t)-w(t_0)$.
	These are the background bounds used in this subsection.
	The local construction and global continuation below use spectral
	energy estimates and Sobolev embeddings. They do not invoke an additional
	Strichartz estimate; the earlier construction of $w$ supplies the stated
	background bounds.
	
	\begin{proposition}[Global well-posedness of the remainder]
		\label{prop:low frequency global revised}
		The initial pair in \eqref{data:low frequency revised} satisfies
		\begin{align}\label{reg:low frequency initial revised}
			v_0\in\dot H^1_{D,rad}(\Omega)\cap L^4(\Omega),\qquad
			\Delta_\Omega v_0\in L^2(\Omega),\qquad v_1\in L^2_{rad}(\Omega),
		\end{align}
		where $\Delta_\Omega v_0$ is understood spectrally.
		Equation~\eqref{eq:low frequency revised} has a unique global energy
		mild solution with
		\begin{align}\label{reg:low frequency energy revised}
			&v\in C(\R;\dot H^1_{D,rad}(\Omega)\cap L^4(\Omega)),
			\qquad\partial_tv\in C(\R;L^2_{rad}(\Omega)),\nonumber\\
			&v-v_0\in C(\R;H^1_{D,rad}(\Omega)).
		\end{align}
		Define
		\begin{align}\label{def:low frequency energy revised}
			E(v)(t)=\frac12\|\partial_tv(t)\|_{L^2}^2
			+\frac12\|\nabla v(t)\|_{L^2}^2
			+\frac14\|v(t)\|_{L^4}^4,
		\end{align}
		and let
		\begin{align*}
			B=\sup_{t\in\R}\|w(t)\|_{\dot H_D^{1/2}}.
		\end{align*}
		Then $E(v)$ is locally absolutely continuous and
		\begin{align}\label{est:low frequency logarithmic bound revised}
			\sqrt{\log(3+E(v)(t))}
			\leq\sqrt{\log(3+E(v)(0))}+C(B+B^2)|t|,
			\qquad t\in\R.
		\end{align}
		In addition,
		\begin{align}\label{reg:low frequency s revised}
			(v,\partial_tv)\in
			C(\R;\dot H^s_{D,rad}(\Omega)\times\dot H^{s-1}_{D,rad}(\Omega)).
		\end{align}
		The boundary condition for the energy solution is the zero Dirichlet trace.
	\end{proposition}
	
	\begin{proof}
		All functions are radial, and spatial norms are taken over $\Omega$.
		
		\noindent\textit{Step 1: Detailed estimates for the low-frequency data.}
		By definition,
		\begin{align*}
			\mathcal F_Dv_j(\lambda)
			=\phi(\lambda/2^J)\mathcal F_Du_j(\lambda),\qquad j=0,1.
		\end{align*}
		The multiplier has absolute value at most one and vanishes for
		$\lambda\geq2^{J+1}$.
		On its support, since $s\leq1$,
		\begin{align*}
			\lambda^2&\leq2^{2(J+1)(1-s)}\lambda^{2s},\\
			\lambda^4&\leq2^{2(J+1)(2-s)}\lambda^{2s},\\
			1&\leq2^{2(J+1)(1-s)}\lambda^{2(s-1)}.
		\end{align*}
		The last inequality is applied for $\lambda>0$.
		The spectral norm formulas consequently give
		\begin{align*}
			\|v_0\|_{\dot H_D^1}^2
			&=4\pi\int_0^{2^{J+1}}\lambda^2
			|\phi(\lambda/2^J)\mathcal F_Du_0(\lambda)|^2\,d\lambda\\
			&\leq2^{2(J+1)(1-s)}\|u_0\|_{\dot H_D^s}^2,\\
			\|\Delta_\Omega v_0\|_2^2
			&=4\pi\int_0^{2^{J+1}}\lambda^4
			|\phi(\lambda/2^J)\mathcal F_Du_0(\lambda)|^2\,d\lambda\\
			&\leq2^{2(J+1)(2-s)}\|u_0\|_{\dot H_D^s}^2,\\
			\|v_1\|_2^2
			&=4\pi\int_0^{2^{J+1}}
			|\phi(\lambda/2^J)\mathcal F_Du_1(\lambda)|^2\,d\lambda\\
			&\leq2^{2(J+1)(1-s)}\|u_1\|_{\dot H_D^{s-1}}^2.
		\end{align*}
		Here $\Delta_\Omega v_0$ denotes the spectral expression with transform
		$-\lambda^2\phi(\lambda/2^J)\mathcal F_Du_0(\lambda)$, which the estimate
		shows to be in $L^2$. This does not require $v_0$ to lie in the
		$L^2$ operator domain of $\Delta_\Omega$.
		Taking square roots yields
		\begin{align}\label{est:low frequency initial revised}
			\|v_0\|_{\dot H_D^1}
			&\leq2^{(J+1)(1-s)}\|u_0\|_{\dot H_D^s},\nonumber\\
			\|\Delta_\Omega v_0\|_2
			&\leq2^{(J+1)(2-s)}\|u_0\|_{\dot H_D^s},\nonumber\\
			\|v_1\|_2
			&\leq2^{(J+1)(1-s)}\|u_1\|_{\dot H_D^{s-1}}.
		\end{align}
		
		The homogeneous energy Sobolev inequality gives
		\begin{align}\label{est:low frequency initial L6 detailed}
			\|v_0\|_{L^6}
			\lesssim\|\nabla v_0\|_2=\|v_0\|_{\dot H_D^1}
			\lesssim2^{(J+1)(1-s)}\|u_0\|_{\dot H_D^s}.
		\end{align}
		This inequality follows by zero extension and the whole-space
		three-dimensional Sobolev inequality for smooth compactly supported
		functions, followed by completion in the gradient norm.
		Proposition~\ref{prop:radial Sob embedding} also gives
		\begin{align*}
			\|v_0\|_4\lesssim_s\|v_0\|_{\dot H_D^s}
			\leq\|u_0\|_{\dot H_D^s}.
		\end{align*}
		Thus the initial energy is finite. With
		$M_s=\|u_0\|_{\dot H_D^s}+\|u_1\|_{\dot H_D^{s-1}}$,
		\begin{align}\label{est:low frequency initial energy revised}
			E(v)(0)\lesssim_s2^{2(J+1)(1-s)}M_s^2+M_s^4.
		\end{align}
		Finite energy does not imply $v_0\in L^2$, since its definition contains
		no $\|v_0\|_2^2$ term. We impose no such additional assumption.
		
		\noindent\textit{Step 2: Local construction of the increment by energy estimates.}
		Set $h=v-v_0$. Then $h(0)=0$, $\partial_th(0)=v_1$, and
		\begin{equation}\label{eq:low frequency affine revised}
			\begin{cases}
				\partial_t^2h-\Delta_\Omega h
				=\Delta_\Omega v_0-\mathcal N(t,h),\\
				h(0)=0,\qquad\partial_th(0)=v_1,
			\end{cases}
		\end{equation}
		where
		\begin{align*}
			\mathcal N(t,h)
			=(v_0+h)^3+3(v_0+h)^2w(t)+3(v_0+h)w(t)^2.
		\end{align*}
		The unknown $h$ is the increment of $v$ relative to its fixed initial
		profile. This auxiliary substitution does not change the nonlinear
		high--low decomposition.
		
		\smallskip
		\noindent\textit{The source takes values in $L^2$.}
		The critical solution satisfies $w\in C(\R;L^6)$ by
		Proposition~\ref{prop:critical sqrt p}.
		For $h\in H_D^1$, we have $v_0+h\in L^6$ by
		\eqref{est:low frequency initial L6 detailed} and $H_D^1\hookrightarrow L^6$.
		H\"older's inequality gives
		\begin{align*}
			\|(v_0+h)^3\|_2&=\|v_0+h\|_6^3,\\
			\|(v_0+h)^2w\|_2&\leq\|v_0+h\|_6^2\|w\|_6,\\
			\|(v_0+h)w^2\|_2&\leq\|v_0+h\|_6\|w\|_6^2.
		\end{align*}
		On $[0,1]$, let $W=\sup_{0\leq t\leq1}\|w(t)\|_6$.
		For $\|h\|_{H_D^1},\|\widetilde h\|_{H_D^1}\leq R$, polynomial
		factorization yields
		\begin{align}\label{est:low frequency local source revised}
			\|\mathcal N(t,h)\|_2
			&\lesssim(\|v_0\|_6+R+W)^3,\nonumber\\
			\|\mathcal N(t,h)-\mathcal N(t,\widetilde h)\|_2
			&\lesssim(\|v_0\|_6+R+W)^2\|h-\widetilde h\|_{H_D^1}.
		\end{align}
		The same bounds show continuity of the source in time.
		
		\smallskip
		\noindent\textit{The linear energy estimate includes an $L^2$ bound.}
		For $F\in L^2$ and $|t|\leq1$, spectral calculus gives
		\begin{align*}
			\left\|\frac{\sin(t\sqrt{-\Delta_\Omega})}{\sqrt{-\Delta_\Omega}}F
			\right\|_2&\leq |t|\|F\|_2,\\
			\left\|\frac{\sin(t\sqrt{-\Delta_\Omega})}{\sqrt{-\Delta_\Omega}}F
			\right\|_{\dot H_D^1}&\leq\|F\|_2,\\
			\|\cos(t\sqrt{-\Delta_\Omega})F\|_2&\leq\|F\|_2.
		\end{align*}
		In particular, combining the first two bounds gives the explicit
		nonhomogeneous estimate
		\begin{align*}
			\left\|\frac{\sin(t\sqrt{-\Delta_\Omega})}{\sqrt{-\Delta_\Omega}}F
			\right\|_{H_D^1}^2
			&=4\pi\int_0^\infty(1+\lambda^2)
			\frac{|\sin(t\lambda)|^2}{\lambda^2}
			|\mathcal F_DF(\lambda)|^2\,d\lambda\\
			&\leq(1+t^2)\|F\|_2^2.
		\end{align*}
		The two contributions use $|\sin(t\lambda)|/\lambda\leq|t|$
		and $|\sin(t\lambda)|\leq1$, respectively.
		This bound controls the $L^2$ norm of the increment as well as its gradient.
		Together with the corresponding bounds for the homogeneous terms, they
		imply, for $0<T\leq1$ and $z_{tt}-\Delta_\Omega z=F$,
		\begin{align}\label{est:low frequency affine linear revised}
			&\|z\|_{C([0,T];H_D^1)}
			+\|\partial_tz\|_{C([0,T];L^2)}\nonumber\\
			&\qquad\leq C\left(\|z(0)\|_{H_D^1}+\|\partial_tz(0)\|_2
			+\int_0^T\|F(\tau)\|_2\,d\tau\right).
		\end{align}
		The propagators are strongly continuous in the indicated spaces.
		
		\smallskip
		\noindent\textit{Picard iteration and the choice of time.}
		Define
		\begin{align}\label{eq:low frequency increment Picard detailed}
			h^{(0)}(t)&=\frac{\sin(t\sqrt{-\Delta_\Omega})}{\sqrt{-\Delta_\Omega}}v_1,
			\nonumber\\
			h^{(n+1)}(t)&=h^{(0)}(t)
			+\int_0^t\frac{\sin((t-\tau)\sqrt{-\Delta_\Omega})}
			{\sqrt{-\Delta_\Omega}}
			\bigl(\Delta_\Omega v_0-\mathcal N(\tau,h^{(n)}(\tau))\bigr)\,d\tau.
		\end{align}
		Choose a fixed constant $K\geq1$ large enough to dominate the constants
		in the preceding linear and nonlinear estimates, and set
		\begin{align*}
			R=2K\|v_1\|_2+1,\qquad L=1+\|v_0\|_6+R+W.
		\end{align*}
		Take
		\begin{align}\label{choice:low frequency local time detailed}
			T=\min\left\{1,
			\frac{R}{4K(1+\|\Delta_\Omega v_0\|_2+L^3)},
			\frac1{4KL^2}\right\}>0.
		\end{align}
		Work in the complete closed ball
		\begin{align*}
			X=\{h\in C([0,T];H^1_{D,rad}(\Omega)):
			\|h\|_{C_tH_D^1}\leq R\}.
		\end{align*}
		For an iterate in this ball,
		\begin{align*}
			\|h^{(n+1)}\|_{C_tH_D^1}
			&\leq K\|v_1\|_2+KT(\|\Delta_\Omega v_0\|_2+L^3)\\
			&\leq R/2+R/4<R.
		\end{align*}
		Since $\|h^{(0)}\|_{C_tH_D^1}\leq K\|v_1\|_2<R$, every iterate remains
		in the same ball. For $n\geq1$, subtraction of consecutive formulas gives
		\begin{align*}
			\|h^{(n+1)}-h^{(n)}\|_{C_tH_D^1}
			\leq KTL^2\|h^{(n)}-h^{(n-1)}\|_{C_tH_D^1}
			\leq\tfrac14\|h^{(n)}-h^{(n-1)}\|_{C_tH_D^1}.
		\end{align*}
		Hence the iterates converge geometrically in $C([0,T];H_D^1)$.
		
		Their time derivatives satisfy
		\begin{align*}
			\partial_th^{(n+1)}(t)
			={}&\cos(t\sqrt{-\Delta_\Omega})v_1\\
			&+\int_0^t\cos((t-\tau)\sqrt{-\Delta_\Omega})
			\bigl(\Delta_\Omega v_0-\mathcal N(\tau,h^{(n)}(\tau))\bigr)\,d\tau.
		\end{align*}
		For $n\geq1$, their successive differences are bounded in $C_tL^2$
		by $KTL^2\|h^{(n)}-h^{(n-1)}\|_{C_tH_D^1}$ and are therefore summable.
		The limits give
		\begin{align}\label{reg:low frequency increment detailed}
			(h,\partial_th)\in C([0,T];H^1_{D,rad}(\Omega)\times L^2_{rad}(\Omega)).
		\end{align}
		The convergence of the source terms permits passage to the limit
		in the Duhamel formula and the corresponding time-derivative formula.
		The limit of the time derivatives is the distributional time
		derivative of $h$, and the initial conditions are attained.
		For the prescribed function $w$ and the fixed profile $v_0$,
		the same difference estimates give local uniqueness.
		They also yield locally Lipschitz dependence on the initial
		state $(h(0),\partial_t h(0))$ in $H_D^1\times L^2$,
		on a common sufficiently short interval for initial states
		in a bounded set.
		
		Since $v=v_0+h$, $v_0\in\dot H_D^1\cap L^4$, and
		$H_D^1\hookrightarrow L^4\cap L^6$, we obtain the local regularity in
		\eqref{reg:low frequency energy revised}.
		In particular, $\partial_tv=\partial_th\in C_tL^2$.
		The identity
		\begin{align*}
			h(t)=\int_0^t\partial_tv(\tau)\,d\tau\quad\text{in }L^2
		\end{align*}
		then explains why the increment is square integrable even though the
		fixed profile $v_0$ need not be. This identity is a consequence of the
		construction, not an additional assumption on $v_0$.
		The backward construction is identical.
		
		At an arbitrary restart time $t_0$, replace the free term in
		\eqref{eq:low frequency increment Picard detailed} by
		\begin{align*}
			\cos((t-t_0)\sqrt{-\Delta_\Omega})h(t_0)
			+\frac{\sin((t-t_0)\sqrt{-\Delta_\Omega})}{\sqrt{-\Delta_\Omega}}
			\partial_th(t_0),
		\end{align*}
		keeping the same fixed profile $v_0$ in the source.
		The same construction applies to any pair in $H_D^1\times L^2$.
		On a fixed finite time interval, bounds on this pair, on
		$\|v_0\|_6$, $\|\Delta_\Omega v_0\|_2$, and on $\|w\|_{L^\infty_tL^6_x}$
		give a common positive restart time. This is the continuation principle
		used in Step 5.
		
		\medskip
		\noindent\textit{Step 3: The energy identity and its justification.}
		Let $I$ be the local existence interval. By the local theory,
		\begin{align*}
			v&\in C\bigl(I;\dot H_D^1(\Omega)\cap L^4(\Omega)\bigr),\\
			\partial_tv&\in C\bigl(I;L^2(\Omega)\bigr).
		\end{align*}
		The Sobolev embedding
		$\dot H_D^1(\Omega)\hookrightarrow L^6(\Omega)$ gives
		$v\in C(I;L^6(\Omega))$. We also have
		$w\in C(I;L^6(\Omega))$.
		Put
		\begin{align*}
			G=-v^3-3v^2w-3vw^2.
		\end{align*}
		H\"older's inequality gives
		\begin{align*}
			\|G(t)\|_2
			&\leq \|v(t)\|_6^3
			+3\|v(t)\|_6^2\|w(t)\|_6
			+3\|v(t)\|_6\|w(t)\|_6^2.
		\end{align*}
		The same product estimates applied to differences imply that
		$G\in C(I;L^2(\Omega))$. Moreover,
		\begin{align*}
			\partial_t^2v-\Delta_\Omega v=G
		\end{align*}
		in the distributional sense.
		
		We claim that $E(v)\in W^{1,1}_{\mathrm{loc}}(I)$ and
		\begin{align}\label{id:low frequency energy revised}
			\frac{d}{dt}E(v)(t)
			=-3\int_\Omega\partial_tv\,v^2w\,dx
			-3\int_\Omega\partial_tv\,vw^2\,dx
		\end{align}
		for almost every $t\in I$.
		Here $W^{1,1}_{\mathrm{loc}}(I)$ denotes scalar functions whose
		distributional time derivative, as well as the function itself,
		is integrable on every compact subinterval of $I$.
		Such functions have locally absolutely continuous
		representatives. Since $E(v)$ is continuous by the local theory,
		its continuous representative is the one used in this identity.
		To justify this identity at the available regularity, fix an arbitrary
		compact interval $[a,b]\subset I$ and treat the quadratic and potential
		parts separately.
		
		\noindent\textit{The quadratic part.}
		Since $v$ need not belong to $L^2(\Omega)$, we first apply a spectral
		cutoff away from both zero and infinity. For each integer $n\geq2$,
		define the spectral projection by the bounded Borel functional calculus:
		
		\begin{align*}
			Q_n=\mathbf1_{[1/n,n]}(\sqrt{-\Delta_\Omega}),
			\qquad v_n=Q_nv.
		\end{align*}
		Here $\mathbf1_{[1/n,n]}$ is the characteristic function of
		$[1/n,n]$, so $Q_n$ keeps precisely the spectral frequencies
		$1/n\leq\lambda\leq n$. On $L^2(\Omega)$, it is an orthogonal
		projection. Its definition uses the spectral theorem and does
		not require the multiplier to be smooth.
		In terms of the distorted Fourier transform, this means
		\begin{align*}
			\mathcal F_Dv_n(t,\lambda)
			=\mathbf1_{[1/n,n]}(\lambda)\mathcal F_Dv(t,\lambda).
		\end{align*}
		The definition makes sense for $v\in\dot H_D^1(\Omega)$ without
		requiring $v\in L^2(\Omega)$. Indeed,
		\begin{align*}
			\|v_n(t)\|_2^2
			&=4\pi\int_{1/n}^n
			|\mathcal F_Dv(t,\lambda)|^2\,d\lambda\\
			&\leq n^2\,4\pi\int_0^\infty
			\lambda^2|\mathcal F_Dv(t,\lambda)|^2\,d\lambda\\
			&=n^2\|v(t)\|_{\dot H_D^1}^2.
		\end{align*}
		Similarly,
		\begin{align*}
			\|(-\Delta_\Omega)v_n(t)\|_2^2
			&=4\pi\int_{1/n}^n
			\lambda^4|\mathcal F_Dv(t,\lambda)|^2\,d\lambda\\
			&\leq n^2\,4\pi\int_0^\infty
			\lambda^2|\mathcal F_Dv(t,\lambda)|^2\,d\lambda\\
			&=n^2\|v(t)\|_{\dot H_D^1}^2.
		\end{align*}
		Thus $v_n(t)$ belongs to the $L^2$ operator domain of
		$-\Delta_\Omega$.
		
		Applying $Q_n$ to the mild equation yields
		\begin{align*}
			\partial_t^2v_n-\Delta_\Omega v_n=Q_nG,
			\qquad
			\partial_tv_n=Q_n\partial_tv.
		\end{align*}
		For each fixed $n$, all spectral multipliers involved are bounded
		on the spectral subspace corresponding to $[1/n,n]$. Consequently,
		$v_n$ has sufficient time and spatial regularity to differentiate
		its quadratic energy. Using self-adjointness of $-\Delta_\Omega$,
		we obtain
		\begin{align*}
			&\frac{d}{dt}
			\left(
			\frac12\|\partial_tv_n(t)\|_2^2
			+\frac12\|\sqrt{-\Delta_\Omega}\,v_n(t)\|_2^2
			\right)\\
			&\qquad=
			\left\langle
			\partial_t^2v_n(t)+(-\Delta_\Omega)v_n(t),
			\partial_tv_n(t)
			\right\rangle_{L^2}\\
			&\qquad=
			\left\langle Q_nG(t),Q_n\partial_tv(t)\right\rangle_{L^2}.
		\end{align*}
		After integration over $[a,b]$, this gives
		\begin{align}\label{id:quadratic cutoff detailed}
			&\frac12\|\partial_tv_n(b)\|_2^2
			+\frac12\|\sqrt{-\Delta_\Omega}\,v_n(b)\|_2^2
			\nonumber\\
			&\quad-
			\frac12\|\partial_tv_n(a)\|_2^2
			-\frac12\|\sqrt{-\Delta_\Omega}\,v_n(a)\|_2^2
			\nonumber\\
			&\qquad=
			\int_a^b
			\left\langle Q_nG(t),Q_n\partial_tv(t)\right\rangle_{L^2}\,dt.
		\end{align}
		
		The spectral projections $Q_n$ are contractions on $L^2(\Omega)$
		and converge strongly to the identity there.
		Since spectral multipliers commute,
		\begin{align*}
			\sqrt{-\Delta_\Omega}\,v_n
			=Q_n\sqrt{-\Delta_\Omega}\,v.
		\end{align*}
		Hence
		\begin{align*}
			\sqrt{-\Delta_\Omega}\,v_n
			&\longrightarrow\sqrt{-\Delta_\Omega}\,v,\\
			\partial_tv_n&\longrightarrow\partial_tv,
			\qquad
			Q_nG\longrightarrow G
		\end{align*}
		in $C([a,b];L^2(\Omega))$.
		
		To explain the uniformity in time, let
		$z\in C([a,b];L^2(\Omega))$. Its image is compact in $L^2(\Omega)$.
		Given $\delta>0$, choose finitely many points $z_1,\ldots,z_m$
		such that each $z(t)$ is within distance $\delta$ of one of these
		points. Then
		\begin{align*}
			\sup_{t\in[a,b]}\|(Q_n-I)z(t)\|_2
			\leq 2\delta
			+\max_{1\leq j\leq m}\|(Q_n-I)z_j\|_2.
		\end{align*}
		Strong convergence on this finite set, followed by $\delta\to0$,
		proves the asserted uniform convergence.
		
		In particular,
		\begin{align*}
			&\left|
			\int_a^b
			\left(
			\langle Q_nG,Q_n\partial_tv\rangle_{L^2}
			-\langle G,\partial_tv\rangle_{L^2}
			\right)dt
			\right|\\
			&\quad\leq
			(b-a)\|Q_nG-G\|_{C_tL^2}\|\partial_tv\|_{C_tL^2}\\
			&\qquad+
			(b-a)\|G\|_{C_tL^2}
			\|Q_n\partial_tv-\partial_tv\|_{C_tL^2}
			\longrightarrow0.
		\end{align*}
		Passing to the limit in \eqref{id:quadratic cutoff detailed} and using
		\begin{align*}
			\|\sqrt{-\Delta_\Omega}\,v(t)\|_2=\|\nabla v(t)\|_2,
		\end{align*}
		we conclude that
		\begin{align}\label{id:quadratic integrated detailed}
			&\frac12\|\partial_tv(b)\|_2^2
			+\frac12\|\nabla v(b)\|_2^2
			-\frac12\|\partial_tv(a)\|_2^2
			-\frac12\|\nabla v(a)\|_2^2
			\nonumber\\
			&\qquad=
			\int_a^b\int_\Omega G\,\partial_tv\,dx\,dt.
		\end{align}
		
		\medskip
		\noindent\textit{The potential part.}
		We justify the chain rule for the potential energy by time
		mollification. Choose $d>0$ such that
		\begin{align*}
			[a-d,b+d]\subset I.
		\end{align*}
		Let $\eta\in C_c^\infty((-1,1))$ satisfy
		\begin{align*}
			\eta\geq0,\qquad
			\int_{\R}\eta(\tau)\,d\tau=1,
		\end{align*}
		and define
		\begin{align*}
			\eta_\varepsilon(\tau)
			=\frac1\varepsilon\eta\left(\frac{\tau}{\varepsilon}\right).
		\end{align*}
		For $0<\varepsilon<d$ and $t\in[a,b]$, set
		\begin{align*}
			v_\varepsilon(t)
			=\int_{-\varepsilon}^{\varepsilon}\eta_\varepsilon(\tau)v(t-\tau)\,d\tau.
		\end{align*}
		Since $\eta_\varepsilon$ is supported in $[-\varepsilon,\varepsilon]$,
		this integral uses only values of $v$ in the larger interval
		$[a-d,b+d]$. It is a well-defined Banach-space integral in
		$L^4(\Omega)\cap L^6(\Omega)$.
		
		Differentiating the convolution kernel gives
		\begin{align*}
			\partial_tv_\varepsilon(t)
			=\int_{a-d}^{b+d}\eta_\varepsilon'(t-\tau)v(\tau)\,d\tau,
		\end{align*}
		The kernel vanishes near the endpoints of the larger interval
		for $t$ in a neighborhood of $[a,b]$.
		Higher time derivatives are obtained in the same way.
		Thus $v_\varepsilon$ is smooth in time, locally near $[a,b]$,
		with values in $L^4(\Omega)\cap L^6(\Omega)$.
		
		On the other hand, convolution commutes with distributional
		differentiation, so
		\begin{align*}
			\partial_tv_\varepsilon(t)
			=\int_{-\varepsilon}^{\varepsilon}\eta_\varepsilon(\tau)
			\partial_tv(t-\tau)\,d\tau
		\end{align*}
		as an $L^2(\Omega)$-valued identity.
		This identity does not require $v\in L^2(\Omega)$.
		Indeed, in the local construction $v=v_0+h$, with
		$h\in C_tH_D^1$ and $\partial_th=\partial_tv\in C_tL^2$;
		the time-independent term $v_0$ disappears upon differentiation.
		
		Since $\int_{\R}\eta_\varepsilon=1$, for $r=4,6$ we have
		\begin{align*}
			v_\varepsilon(t)-v(t)
			=\int_{-\varepsilon}^{\varepsilon}\eta_\varepsilon(\tau)
			\bigl(v(t-\tau)-v(t)\bigr)\,d\tau
		\end{align*}
		and therefore
		\begin{align*}
			\sup_{t\in[a,b]}\|v_\varepsilon(t)-v(t)\|_r
			\leq
			\sup_{\substack{t\in[a,b]\\|\tau|\leq\varepsilon}}
			\|v(t-\tau)-v(t)\|_r
			\longrightarrow0.
		\end{align*}
		The last convergence follows from uniform continuity on the larger
		compact interval. Applying the same argument to $\partial_tv$
		gives
		\begin{align}\label{conv:potential mollification detailed}
			v_\varepsilon&\longrightarrow v
			\quad\text{in }C([a,b];L^4\cap L^6),\nonumber\\
			\partial_tv_\varepsilon&\longrightarrow\partial_tv
			\quad\text{in }C([a,b];L^2).
		\end{align}
		
		The functional
		\begin{align*}
			\Phi(f)=\frac14\int_\Omega |f(x)|^4\,dx
		\end{align*}
		is continuously Fr\'echet differentiable on the real space
		$L^4(\Omega)$, and
		\begin{align*}
			\Phi'(f)[g]=\int_\Omega f^3g\,dx.
		\end{align*}
		Since $v_\varepsilon$ is differentiable with values in $L^4$,
		the Banach-space chain rule yields
		\begin{align*}
			\frac{d}{dt}\frac14\|v_\varepsilon(t)\|_4^4
			=\int_\Omega v_\varepsilon^3\partial_tv_\varepsilon\,dx.
		\end{align*}
		Integrating from $a$ to $b$, we obtain
		\begin{align}\label{id:potential mollified detailed}
			\frac14\|v_\varepsilon(b)\|_4^4
			-\frac14\|v_\varepsilon(a)\|_4^4
			=\int_a^b\int_\Omega
			v_\varepsilon^3\partial_tv_\varepsilon\,dx\,dt.
		\end{align}
		
		We now pass to the limit. By the algebraic identity for the
		difference of two cubes and H\"older's inequality,
		\begin{align*}
			\|v_\varepsilon^3-v^3\|_2
			\lesssim
			\bigl(\|v_\varepsilon\|_6^2+\|v\|_6^2\bigr)
			\|v_\varepsilon-v\|_6.
		\end{align*}
		The convergence in \eqref{conv:potential mollification detailed}
		therefore implies
		\begin{align*}
			v_\varepsilon^3\longrightarrow v^3
			\quad\text{in }C([a,b];L^2).
		\end{align*}
		Consequently,
		\begin{align*}
			&\left|
			\int_a^b\int_\Omega
			\left(v_\varepsilon^3\partial_tv_\varepsilon
			-v^3\partial_tv\right)\,dx\,dt
			\right|\\
			&\quad\leq
			(b-a)\|v_\varepsilon^3-v^3\|_{C_tL^2}
			\|\partial_tv_\varepsilon\|_{C_tL^2}\\
			&\qquad+
			(b-a)\|v^3\|_{C_tL^2}
			\|\partial_tv_\varepsilon-\partial_tv\|_{C_tL^2}
			\longrightarrow0.
		\end{align*}
		The endpoint terms converge by the $C_tL^4$ convergence. Passing to
		the limit in \eqref{id:potential mollified detailed} gives
		\begin{align}\label{id:potential integrated detailed}
			\frac14\|v(b)\|_4^4-\frac14\|v(a)\|_4^4
			=\int_a^b\int_\Omega v^3\partial_tv\,dx\,dt.
		\end{align}
		
		\medskip
		\noindent\textit{Conclusion.}
		Adding \eqref{id:quadratic integrated detailed} and
		\eqref{id:potential integrated detailed}, we find
		\begin{align*}
			E(v)(b)-E(v)(a)
			&=\int_a^b\int_\Omega (G+v^3)\partial_tv\,dx\,dt\\
			&=-3\int_a^b\int_\Omega\partial_tv\,v^2w\,dx\,dt\\
			&\quad-3\int_a^b\int_\Omega\partial_tv\,vw^2\,dx\,dt.
		\end{align*}
		Thus the cubic contribution $-v^3$ in $G$ cancels exactly with
		the derivative of the potential energy.
		
		Finally,
		\begin{align*}
			\left|\int_\Omega\partial_tv\,v^2w\,dx\right|
			&\leq\|\partial_tv\|_2\|v\|_6^2\|w\|_6,\\
			\left|\int_\Omega\partial_tv\,vw^2\,dx\right|
			&\leq\|\partial_tv\|_2\|v\|_6\|w\|_6^2.
		\end{align*}
		Both expressions are bounded on every compact time interval.
		The integrated identity therefore represents $E(v)$ as a constant
		plus the time integral of a locally integrable function.
		It follows that
		\begin{align*}
			E(v)\in W^{1,1}_{\mathrm{loc}}(I),
		\end{align*}
		and \eqref{id:low frequency energy revised} holds for almost every
		$t\in I$. No time derivative of $w$ is used.
		
		\noindent\textit{Step 4: A logarithmic energy estimate.}
		For brevity write $E=E(v)(t)$ at a fixed time.
		The energy and the Sobolev embedding give
		\begin{align}\label{est:low frequency energy norms revised}
			\|\partial_tv\|_2+\|v\|_6+\|v\|_\infty\lesssim E^{1/2},
			\qquad\|v\|_4\lesssim E^{1/4}.
		\end{align}
		The radial $L^\infty$ bound is specific to the exterior region $r\geq1$.
		For smooth compactly supported radial functions,
		\begin{align*}
			|v(r)|
			&\leq\int_r^\infty|\partial_\rho v(\rho)|\,d\rho\\
			&\leq\left(\int_r^\infty|\partial_\rho v(\rho)|^2\rho^2\,d\rho\right)^{1/2}
			\left(\int_r^\infty\rho^{-2}\,d\rho\right)^{1/2}
			\lesssim r^{-1/2}\|\nabla v\|_2\leq C\|\nabla v\|_2.
		\end{align*}
		The bound extends to the radial homogeneous energy space by completion.
		
		Proposition~\ref{prop:critical sqrt p} gives, uniformly for $4\leq p<\infty$,
		\begin{align}\label{est:low frequency w p revised}
			\|w(t)\|_p\leq C\sqrt p\,B.
		\end{align}
		Take any $p\geq6$ and set $q=4p/(p-2)$, so that
		$\frac12+\frac1p+\frac2q=1$.
		Interpolation between $L^4$ and $L^\infty$ yields
		\begin{align*}
			\|v\|_q^2
			\leq\|v\|_4^{2-4/p}\|v\|_\infty^{4/p}
			\lesssim E^{1/2+1/p},
		\end{align*}
		with a constant uniform for $p\geq6$.
		Hence
		\begin{align*}
			\left|\int_\Omega\partial_tv\,v^2w\,dx\right|
			&\leq\|\partial_tv\|_2\|w\|_p\|v\|_q^2
			\lesssim B\sqrt p\,E^{1+1/p},\\
			\left|\int_\Omega\partial_tv\,vw^2\,dx\right|
			&\leq\|\partial_tv\|_2\|v\|_6\|w\|_6^2
			\lesssim B^2E.
		\end{align*}
		Thus \eqref{id:low frequency energy revised} implies
		\begin{align}\label{est:low frequency energy p revised}
			|E'(t)|\leq C B\sqrt p\,E(t)^{1+1/p}+CB^2E(t),\qquad p\geq6.
		\end{align}
		
		Let $Y=3+E$ and choose, at each time,
		\begin{align*}
			p=\max\{6,\log Y\}.
		\end{align*}
		Then $Y^{1/p}\leq e$ and $\sqrt p\leq C\sqrt{\log Y}$.
		All constants in \eqref{est:low frequency energy p revised} are uniform in
		$p$, so this choice gives
		\begin{align*}
			|Y'(t)|\leq C(B+B^2)Y(t)\sqrt{\log Y(t)}.
		\end{align*}
		The chain rule for locally absolutely continuous functions therefore gives
		\begin{align*}
			\left|\frac{d}{dt}\sqrt{\log Y(t)}\right|
			\leq C(B+B^2).
		\end{align*}
		Integration from zero proves
		\eqref{est:low frequency logarithmic bound revised} on the maximal local
		existence interval. In particular,
		\begin{align}\label{est:low frequency explicit energy revised}
			E(v)(t)\leq
			\exp\!\left(\left[\sqrt{\log(3+E(v)(0))}
			+C(B+B^2)|t|\right]^2\right)-3.
		\end{align}
		The energy is therefore bounded on every finite time interval.
		
		\noindent\textit{Step 5: Global continuation of the affine energy solution.}
		Suppose, for contradiction, that its maximal forward endpoint $T_+$ is finite.
		By \eqref{est:low frequency explicit energy revised}, there is a finite
		constant $E_*$ with $E(v)(t)\leq E_*$ for $0\leq t<T_+$.
		Since $h(0)=0$ and $\partial_th=\partial_tv$, the fundamental theorem in
		$L^2$ gives
		\begin{align*}
			\|h(t)\|_2\leq\int_0^t\|\partial_tv(\tau)\|_2\,d\tau
			\leq T_+\sqrt{2E_*}.
		\end{align*}
		Also,
		\begin{align*}
			\|\nabla h(t)\|_2\leq\sqrt{2E_*}+\|\nabla v_0\|_2,
			\qquad\|\partial_th(t)\|_2\leq\sqrt{2E_*}.
		\end{align*}
		Thus, after increasing the finite constant if necessary,
		\begin{align*}
			\sup_{0\leq t<T_+}
			\bigl(\|h(t)\|_{H_D^1}+\|\partial_th(t)\|_2\bigr)
			\leq A_*,\qquad
			A_*=1+(T_++2)\sqrt{2E_*}+\|\nabla v_0\|_2.
		\end{align*}
		The background bound \eqref{est:low frequency background properties revised}
		provides $W_*:=\sup_{t\in\R}\|w(t)\|_6<\infty$.
		To make the continuation argument explicit, choose
		\begin{align*}
			R_*=2K(A_*+1),\qquad L_*=1+\|v_0\|_6+R_*+W_*,
		\end{align*}
		where $K$ is the fixed constant in the local energy construction, and take
		\begin{align*}
			\ell=\min\left\{1,
			\frac{R_*}{4K(1+\|\Delta_\Omega v_0\|_2+L_*^3)},
			\frac1{4KL_*^2}\right\}>0.
		\end{align*}
		At any time $t_0<T_+$, the free restarted increment is bounded by
		$K(\|h(t_0)\|_{H_D^1}+\|\partial_th(t_0)\|_2)\leq KA_*<R_*/2$.
		The same Picard estimates as in Step 2 then construct a solution on
		$[t_0,t_0+\ell]$ in the ball of radius $R_*$.
		All quantities defining $\ell$ are independent of $t_0$.
		Choose $t_0$ so that $T_+-t_0<\ell/2$.
		The restarted solution extends beyond $T_+$ and agrees with the original
		one on their overlap by local uniqueness, contradicting maximality.
		The backward endpoint is handled in the same way.
		This proves global existence and uniqueness in the affine energy class.
		Any energy mild solution in the stated class has
		$v(t)-v_0=\int_0^t\partial_tv(\tau)\,d\tau$ in $L^2$ locally and hence
		belongs to this affine class, so uniqueness holds in
		\eqref{reg:low frequency energy revised} as well.
		
		\noindent\textit{Step 6: Recovery of the $s$-level regularity.}
		Let
		\begin{align*}
			\mathcal Q=v^3+3v^2w+3vw^2.
		\end{align*}
		Since $1-s\in[\frac14,\frac12]$, Proposition~\ref{prop:radial Sob embedding}
		and duality imply
		\begin{align*}
			\|F\|_{\dot H_D^{s-1}}\lesssim_s\|F\|_{L^{4/3}}
		\end{align*}
		for radial $F$.
		On $[-T,T]$, set $E_T=\sup_{|t|\leq T}E(v)(t)<\infty$.
		H\"older's inequality gives
		\begin{align*}
			\|\mathcal Q(t)\|_{\dot H_D^{s-1}}
			&\lesssim_s\|v(t)\|_4^3
			+\|v(t)\|_4^2\|w(t)\|_4+\|v(t)\|_4\|w(t)\|_4^2\\
			&\lesssim_s(E_T^{1/4}+B)^3.
		\end{align*}
		The source is continuous in this space.
		Moreover, $\mathcal Q\in C_tL^2$ by the local energy construction.
		The energy Duhamel formula is
		\begin{align}\label{eq:low frequency mild revised}
			v(t)={}&\cos(t\sqrt{-\Delta_\Omega})v_0
			+\frac{\sin(t\sqrt{-\Delta_\Omega})}{\sqrt{-\Delta_\Omega}}v_1
			\nonumber\\
			&-\int_0^t\frac{\sin((t-\tau)\sqrt{-\Delta_\Omega})}
			{\sqrt{-\Delta_\Omega}}\mathcal Q(\tau)\,d\tau.
		\end{align}
		This formula can also be recovered directly from the increment formula.
		The constant forcing combines with the initial profile according to
		\begin{align*}
			v_0+\int_0^t
			\frac{\sin((t-\tau)\sqrt{-\Delta_\Omega})}{\sqrt{-\Delta_\Omega}}
			\Delta_\Omega v_0\,d\tau
			=\cos(t\sqrt{-\Delta_\Omega})v_0.
		\end{align*}
		On a compact spectral annulus this is the scalar identity
		\begin{align*}
			1-\lambda\int_0^t\sin((t-\tau)\lambda)\,d\tau
			=\cos(t\lambda).
		\end{align*}
		Removal of the cutoffs gives the identity in the homogeneous energy
		space, without assuming $v_0\in L^2$.
		The initial pair belongs to $\dot H_D^s\times\dot H_D^{s-1}$ by
		the spectral cutoff bounds.
		Applying the $s$-level spectral multiplier estimates to
		\eqref{eq:low frequency mild revised} and its derivative gives
		\begin{align}\label{est:low frequency s revised}
			&\sup_{|t|\leq T}\bigl(\|v(t)\|_{\dot H_D^s}
			+\|\partial_tv(t)\|_{\dot H_D^{s-1}}\bigr)
			\nonumber\\
			&\qquad\lesssim_s\|v_0\|_{\dot H_D^s}
			+\|v_1\|_{\dot H_D^{s-1}}
			+T(E_T^{1/4}+B)^3.
		\end{align}
		The energy and $s$-level formulas agree by restriction to compact spectral
		annuli and removal of the cutoffs. Their continuity proves
		\eqref{reg:low frequency s revised}.
	\end{proof}
	
	\begin{corollary}[Recombination of the two components]
		\label{cor:full global recombination revised}
		For radial data in $\dot H_D^s\times\dot H_D^{s-1}$ with
		$\frac12\leq s\leq\frac34$, the function $u=v+w$ constructed above
		is the unique global mild solution of
		\begin{align*}
			\partial_t^2u-\Delta_\Omega u+u^3=0,
			\qquad (u,\partial_tu)|_{t=0}=(u_0,u_1),
		\end{align*}
		with $(u,\partial_tu)\in C(\R;\dot H_D^s\times\dot H_D^{s-1})$.
		It depends locally Lipschitz continuously on the original initial pair
		on each finite time interval.
	\end{corollary}
	
	\begin{proof}
		Adding the equations for $v$ and $w$ gives the cubic equation for $u$,
		and the initial projections sum to the identity.
		The two $s$-level Duhamel formulas therefore sum to the full mild formula.
		Uniqueness follows from Proposition~\ref{prop:general s local Picard}
		on successive local intervals.
		
		To verify continuous dependence, fix the original initial pair and a finite
		$T>0$. Choose $J$ so that its high-frequency critical norm is at most
		half the smallness threshold.
		The bounded spectral map from the original $s$-level data to the truncated
		critical pair allows this same $J$ for all sufficiently close original data.
		Their high-frequency solutions have a common bound for $B$ and common
		finite-time $s$-level bounds by
		Proposition~\ref{prop:high frequency global revised}.
		Their low initial energies have a common bound by
		\eqref{est:low frequency initial energy revised}, and hence the energy bound
		\eqref{est:low frequency explicit energy revised} is uniform on $[-T,T]$.
		Together with \eqref{est:low frequency s revised}, this supplies a common
		finite-time bound on the full $s$-level solution pairs.
		The local Lipschitz estimate in Proposition~\ref{prop:general s local Picard}
		can now be iterated on a common finite partition of $[-T,T]$.
		This proves the asserted local Lipschitz dependence.
		The extra initial condition $u_0\in L^4$ is unnecessary in this range
		by Proposition~\ref{prop:radial Sob embedding}.
	\end{proof}

	\def\cprime{$'$}
	
\end{document}